\documentclass[11pt]{article}
\usepackage{amsmath}
\usepackage{amsfonts}
\usepackage{graphicx}
\usepackage[pdftex]{hyperref}
\usepackage[dutch,english]{babel}
\usepackage{amsmath,amssymb}
\usepackage{amsthm}
\usepackage[usenames,dvipsnames]{color}
\usepackage{tikz}
\usepackage{accents}
\usepackage{xcolor}
\usepackage{hyperref}
\usepackage{mathtools}
\usepackage{multicol}   
\usepackage{comment}
\usepackage{adjustbox}
\usepackage{subcaption}

\usepackage{tikz-cd} 
\usetikzlibrary{babel}

\usepackage{tikz}
\usetikzlibrary{shadows, arrows}

\usepackage{chemfig}

\usepackage{fullpage}
\usepackage[all]{xy}

\newcommand{\bl}[1]{{\color{blue}#1}}

\newcommand{\ora}[1]{{\color{orange}#1}}

\newcommand{\eps}{\varepsilon}

\theoremstyle{plain}
\newtheorem{thr}{Theorem}[section]

\newtheorem{lem}[thr]{Lemma}
\newtheorem{prop}[thr]{Proposition}
\newtheorem{cor}[thr]{Corollary}
\theoremstyle{remk}
\newtheorem{remark1}[thr]{Remark}
\newenvironment{remk}{\begin{remark1}\rm}{\hfill $\triangle$ \end{remark1}}

\theoremstyle{ex}
\newtheorem{ex1}[thr]{Example}
\newenvironment{ex}{\begin{ex1}\rm}{\hfill $\triangle$ \end{ex1}}
\newenvironment{defi}{\begin{defi1}\rm}{\hfill $\triangle$ \end{defi1}}

\theoremstyle{defi}
\newtheorem{defi1}[thr]{Definition}

\definecolor{wred}{rgb}{0.7,0.18,0.12}
\definecolor{wgreen}{rgb}{0.1,0.53,0.37}

\numberwithin{equation}{section}

\def\barray{\begin{eqnarray*}}             \def\earray{\end{eqnarray*}}
\def\beq{\begin{equation}} \def\eeq{\end{equation}}

\title{Multiple timescales and the parametrisation method\\ in geometric singular perturbation theory}
\author{Ian Lizarraga\thanks{School of Mathematics and Statistics, The University of Sydney, {\tt ian.lizarraga@sydney.edu.au}}, Bob Rink\thanks{Department of Mathematics, Vrije Universiteit Amsterdam, The Netherlands, {\tt b.w.rink@vu.nl}.} \ and Martin Wechselberger\thanks{School of Mathematics and Statistics, The University of Sydney, {\tt  martin.wechselberger@sydney.edu.au}.}}
\date{\today}

\excludecomment{percent}

\newcommand\WM[1]{{\color{violet}#1}}

\begin{document}

\maketitle

\abstract{We present a novel method for computing slow manifolds and their fast fibre bundles in geometric singular perturbation problems. This coordinate-independent method is inspired by the parametrisation method introduced by Cabr\'e, Fontich and de la Llave.  By iteratively solving a so-called conjugacy equation, our method simultaneously computes parametrisations of slow manifolds and fast fibre bundles, as well as the dynamics on these objects, to arbitrarily high degrees of accuracy. We show the power of this {\em top-down} method for the study of systems with multiple (i.e., three or more) timescales. In particular, we highlight the emergence of {\em hidden} timescales and show how our method can uncover these surprising multiple timescale structures. We also apply our parametrisation method to several reaction network problems.}


\section{Introduction}

Many evolving biological and physical systems display distinct temporal features, which can be attributed to processes taking place on different timescales. In mathematical terms, such multiple timescale models are in fact {\em singular perturbation problems}. 
In this paper we consider parameter-dependent differential equations of the general form 
\begin{equation}\label{eq:singpert}
x' = \frac{dx}{dt} = F(x,\eps)  =\sum_{i=0}^{l} \eps^i F_i(x)\,\qquad (l\ge 1)\, ,
\end{equation}
for $x\in \mathbb{R}^n$. Here, $0 \leq \varepsilon\ll 1$ is a small perturbation parameter, and the smooth functions $F_i:\mathbb{R}^n\to \mathbb{R}^n$ model processes evolving on distinct timescales $t_i=\eps^i t$. 

In the case that the zero-limit critical set
$$
S:=\{x\in \mathbb{R}^n\,:\,F_0(x)=0\}
$$
contains a $k$-dimensional differentiable {\em critical manifold} $S_0\subseteq S$, with $1\le k\leq n-1$, system \eqref{eq:singpert} is considered a singular perturbation problem. The slower processes  $F_{i} \, (i\geq 1)$ will then determine the dynamics of the model in phase space near $S_0$.

A classical theorem of Fenichel \cite{fenichel1979} guarantees that the (compact) critical manifold  $S_0$ persists as a flow-invariant {\it slow manifold} $S_0^\eps$ in any dynamical system close to the singular limit $x'= F_0(x)$, under the condition that the critical manifold $S_0$ is {\em normally hyperbolic}. The corresponding slow flow on $S_0^\eps$ is then governed to leading order by the {\em reduced problem}
\begin{equation}\label{eq:reduced-singpert}
\dot{x}= P_0(x) F_1(x)\,,\qquad \forall x\in S_0\,,
\end{equation}
where the overdot indicates differentation with respect to the slow time $t_1=\eps t$, and $P_0$ is a certain (unique) projection  onto the {\em tangent bundle} ${\cal T}S_0$. This {\em geometric singular perturbation theory} (GSPT) result due to Fenichel not only describes the persistence of an invariant slow manifold $S_0^\eps$ near $S_0$, but also the (local) persistence of an invariant  nonlinear {\em fast fibre bundle} ${\cal L}S_0^\eps$ that describes the fast dynamics towards or away from the slow manifold $S_0^\eps$.

Under the basic assumption that the critical manifold $S_0$ is the image of a smooth embedding $\phi_0: U_1\to \mathbb{R}^n$ defined on an open subset $U_1\subset \mathbb{R}^k$, the $k$-dimensional reduced leading order problem \eqref{eq:reduced-singpert} can alternatively be described in the (local) coordinate chart $U_1$ by the differential equation
\begin{equation}\label{eq:r}
\dot{\xi}= r_1(\xi) := L_0(\xi)\, P_0 F_1(\phi_0(\xi))\ \ \mbox{for}\ \xi\in U_1\, .
\end{equation}
Here, $L_0(\xi)D\phi_0(\xi)={\rm id}_{\mathbb{R}^k},\,\forall \xi\in U_1$, i.e., $L_0(\xi)$ is a {\em left inverse} of $D\phi_0(\xi)$; see, e.g., Feliu et al \cite{feliu2020}. 
In this paper, we build upon this local representation in a coordinate chart, and we provide an alternative way to compute it, by introducing a {\em parametrisation method} for finding slow invariant manifolds and their linear fast fibre bundles, to any order in the small parameter, and for general coordinate-independent multiple timescale problems of the form \eqref{eq:singpert}. A key observation underlying our method is that the identification of these manifolds and bundles can be reframed in terms of solving a so-called {\it conjugacy equation}. A solution to this conjugacy equation consists of embeddings of the slow manifold and corresponding linear fast fibre bundle, as well as the slow and fast vector fields that they carry (in the appropriate coordinate chart).  Our parametrisation method seeks to iteratively solve the conjugacy equation, thereby calculating the slow and fast dynamics to arbitrary precision.


%

To the best of our knowledge this paper presents the first parametrisation method for singular perturbation problems. At the same time, the idea to compute invariant manifolds by solving a conjugacy equation is of course classical. A well-known example is the Lyapunov-Schmidt method for computing periodic orbits near Hopf bifurcations. In this setting, the problem is phrased as a conjugacy equation on a loop space, which can be solved with the implicit function theorem, see for instance \cite{duisbif, Golstewperspective, vdbauw}. The  idea is also 
very  prominent in the work of Cabr\'e, Fontich and de la Llave, who introduced a parametrisation method in the early 2000's. More precisely, in  \cite{Cabre1, Cabre2, Cabre3} they provide an iterative scheme for computing stable and unstable manifolds of hyperbolic fixed points of maps and flows, and analytically prove the convergence of this scheme. The scope of their method has since been extended considerably, for example to delay equations \cite{delay} and partial differential equations  \cite{PDEparametrisation}, and variants of the method  have been introduced to compute, e.g. KAM tori \cite{ActionAngle} and centre manifolds \cite{wouter}. We refer to \cite{parameterisation_book} for an extensive list of references and an overview of results.

We should mention in particular the work of Haro and Canadell \cite{HaroCanadell}, who introduced a parametrisation method for computing normally hyperbolic invariant manifolds. Their method consists of a Newton iteration, and it was used to numerically compute normally hyperbolic quasi-periodic tori in Hamiltonian systems. For such quasi-periodic tori, their method was proved to converge in \cite{HaroCanadell2}.  
The setting of the current paper is similar in nature: we present a parametrisation method for slow manifolds, which are also normally hyperbolic invariant manifolds, but having the additional property that the flow that they carry is close to stationary. This latter property allows us to provide explicit analytical formulas for the solutions to the homological equations that arise in each step of the iteration, see for example the proofs of Theorem \ref{infconjthm} and Theorem \ref{solveinflinear}. Such analytical formulas are not available for the iterative method in \cite{HaroCanadell, HaroCanadell2}. On the other hand, in contrast with  \cite{HaroCanadell, HaroCanadell2}, the iterative scheme that we provide here is a quasi-Newton method, and hence not quadratically but linearly convergent.

\begin{figure*}[t]
  \centering 
      \includegraphics[trim=10cm 4cm 10cm 2cm,clip,width=9.5cm,height=11.5cm]{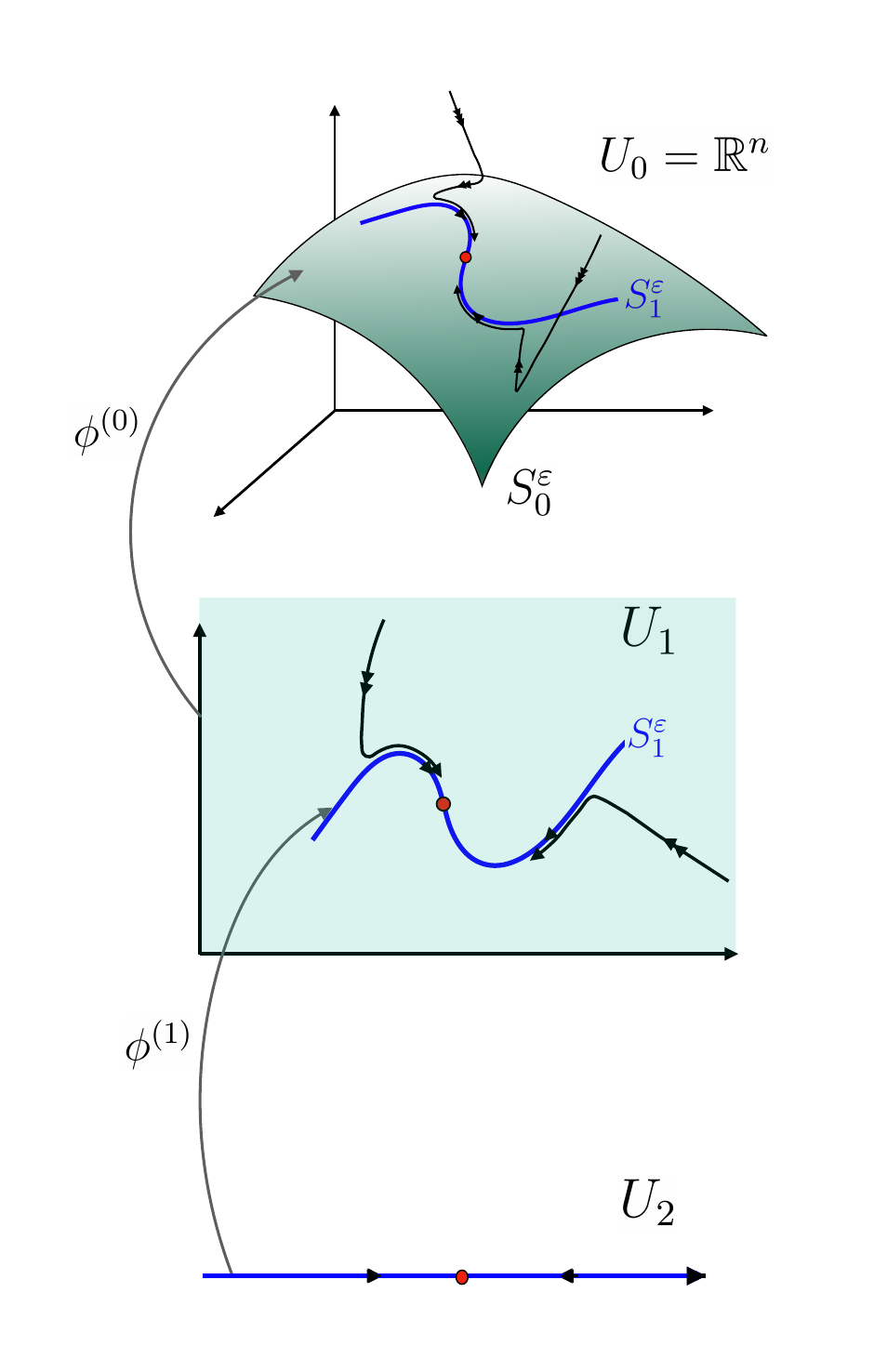}       

      \caption{Sketch of a multiple timescale dynamical system having two nested slow invariant manifolds $S_0^{\eps}$ and $S_1^{\eps}$. The coordinate chart $U_1$ is embedded as $S_0^{\eps}\subset U_0 = \mathbb{R}^n$ by $\phi^{(0)}$, and the coordinate chart $U_2$  is embedded as $S_1^{\eps}\subset U_1$ by $\phi^{(1)}$ (this superscript notation $^{(i)}$ will be introduced in section~\ref{secmultiple}). The red point depicts the possibility of a lower-dimensional invariant object (e.g.~an equilibrium point, or a further submanifold introducing yet another dynamical timescale). The parametrisation method introduced in this article uses a {\em top-down} approach, i.e.~the slow manifold $S_0^\eps$ is identified first, then the infra-slow manifold $S_1^\eps$, etc.}
            \label{fig:embedding}
\end{figure*}

The parametrisation method becomes particularly useful when dealing with genuine singularly perturbed multiple timescale problems in which the slow flow on the slow manifold $S_0^\eps$ is a singular perturbation problem itself, i.e., if there exists an embedded lower-dimensional infra-slow invariant manifold $S_1^\eps\subset S_0^\eps$ where processes on the timescale $t_2=\eps^2 t$ become dominant.\footnote{Here we are identifying $S_1^{\eps} \subset U_1$ with its image $\phi_0(S_1^{\eps})\subset \mathbb{R}^n$ under the smooth embedding $\phi_0$. Such natural identifications will be implicitly applied throughout the paper whenever we discuss nesting of the slow manifolds as subsets in $\mathbb{R}^n$.}  Note that the equation $\dot{\xi} = r_1(\xi) + \mathcal{O}(\eps)$ on the coordinate chart $U_1$ is itself a singular perturbation problem when $S_1 = \{\xi\in U_1 \, : \, r_1(\xi)=0\}$ contains a critical manifold of dimension $1\leq n_2 \leq k-1$. This simple observation will lead to a geometric definition of a singularly perturbed multiple timescale system, possessing nested invariant (slow, infra-slow,\ldots) manifolds $S_0^\eps\supset S_1^\eps\supset\cdots\supset S_{m-2}^\eps$ that support flows evolving on distinct (slow, infra-slow,\ldots) timescales $t_1=\eps t,\,t_2=\eps^2 t\,,\ldots,t_{m-1}=\eps^{m-1}t$.
%
We provide algorithms to calculate these nested slow manifolds, and their slow flows, 
using a {\em top-down} approach: our method calculates the `top' slow manifold $S_0^\eps $ first, and then makes its way `down' the chain of nested slow manifolds; see Figure~\ref{fig:embedding}.
%
%

First theoretical results on multiple timescale problems have been provided by
Cardin \& Teixeira \cite{cardin2017} who 
studied multiple timescale problems given in the {\em standard form}
\begin{equation}\label{eq:cardin1}
\begin{aligned}
x_j' & = \left(\prod_{i=0}^{m-j} \eps_i\right)  f_j(x,\eps)\,,\quad j=1,\ldots,m \, ,
\end{aligned}
\end{equation}
where $x=(x_1,\ldots,x_m)\in\mathbb{R}^n$, $m\ge 2$, $x_j\in\mathbb{R}^{l_j}$, $\sum_{j=1}^m l_j=n$, $\eps_0=1$ and $\eps=(\eps_1,\ldots,\eps_m)$ is a vector of {\em independent} small parameters $\eps_1,\ldots,\eps_m\ll 1$. In this standard framework, the underlying multiple timescale processes are identified with specific coordinates: $x_m$ is a fast variable, $x_{m-1}$ is a slow variable, \ldots, $x_1$ is the slowest variable. Here, one may also identify the nested invariant (critical) manifolds $S_0\supset S_1\supset \cdots\supset S_{m-2}$ {\em a-priori} as the zero sets
\begin{equation}\label{def:nested-standard}
S_{l}=\{x\in\mathbb{R}^n: f_j(x,0)=0,\; j=1,\ldots l+1\}\,,\quad l=0,\ldots m-2\,.
\end{equation}
Assuming that these nested  critical manifolds $S_l$ are all normally hyperbolic\footnote{viewed in the corresponding system evolving on the timescale $t_{l+1}=\eps^{l+1}t$.}, the authors of \cite{cardin2017}   show 
the persistence of nested invariant (slow, infra-slow,\ldots) manifolds $S_0^\eps\supset S_1^\eps\supset \cdots\supset S_{m-2}^\eps$. In contrast to our proposed top-down approach, their proof follows a  {\em bottom-up} approach, i.e., they first show the existence of the `bottom' manifold $S_{m-2}^\eps$ based on Fenichel theory, and then work their way `up' to the top manifold $S_{0}^\eps$. We also note that the independence of the small parameters $\eps=(\eps_1,\ldots,\eps_m)$ is crucial in their proof.
Recently, Kruff \& Walcher \cite{kruff2019} provided first results on three timescale problems in a nonstandard, coordinate independent setting. Their approach is based on transforming a coordinate-independent system into a (local) standard form first, and then applying the results of Cardin \& Teixeira.


We would like to point out that the nested sets given by \eqref{def:nested-standard} are restrictive when it comes to identifying multiple timescale structures in singular perturbation problems,
i.e.~the {\em bottom-up} approach relies on {\it a priori} knowledge of ultimate attractor states and/or the number of timescales involved in the dynamics.
%
%
In the following example, we highlight some of the subtleties that may arise in general multiple timescale dynamics.

%

\begin{ex}
Let us consider the following (reaction) network motif
$$
\xymatrix{
& X_2 \ar@{-|}[dr] &\\
X_3 \ar[rr] \ar[ur]& &X_1
}
$$
describing an {\em incoherent feed-forward loop} (IFFL) with three reactants; see e.g.~\cite{alon2007}.
In dimensionless form, the evolution equations for the reactant concentrations that govern a very simple form of such an IFFL,  are
\begin{align}\label{babyreaction}
\frac{d}{dt}\left( \begin{array}{c} 
 x_1 \\
x_2 \\ 
x_3
\end{array} \right) = 
\underbrace{ \left( \begin{array}{c} 0 \\
0 \\ 
- x_3
\end{array} \right)}_{F_0(x)} + 
\varepsilon \underbrace{ \left( \begin{array}{c}
a_1 x_3 - a_2x_1 x_2  \\
  a_3 x_3 - a_4 x_2\\
1 
\end{array} \right)}_{F_1(x)} \, 
 \,.
\end{align}
Here, $0 < \varepsilon \ll 1$ is a small parameter, i.e., we assume that the decay rate of $x_3$ is significantly faster than all other reaction rates, but we assume that all other parameters $a_i=O(1)$, $i=1,\ldots,4$, are of the same order. 

Firstly, we notice that this reaction network is governed by two processes,
one fast process $F_0$ and one slow process $F_1$. System \eqref{babyreaction} is actually of the standard form \eqref{eq:cardin1}, i.e., $x_3$ is considered a fast variable while $(x_1,x_2)$ are slow variables. One may therefore naively expect  to deal with a standard two timescale problem. It turns out that this intuition is incorrect. 

To see this, note that the critical manifold $S_0=\{(x_1,x_2,x_3)\in\mathbb{R}^3: x_3=0\}$ is a two-dimensional plane.
 In fact, because $\frac{dx_3}{dt} = -x_3 + \varepsilon$, every solution converges exponentially to the nearby invariant slow manifold $S_0^\eps=\{(x_1,x_2,x_3)\in\mathbb{R}^3: x_3=\eps\}$.
The slow dynamics on $S_0^\eps$ is governed by the two-dimensional system
\begin{align}\label{baby-slow}
\frac{d}{dt_1}\left( \begin{array}{c} \ x_1 \\
 x_2
\end{array} \right) = \underbrace{
 \left( \begin{array}{c} - a_2 x_1 x_2 \\
- a_4 x_2
\end{array} \right) }_{r_1(x_1, x_2)}+ 
\varepsilon \underbrace{\left( \begin{array}{c}
a_1    \\
  a_3 
\end{array} \right)}_{r_2(x_1, x_2)}\, ,
\end{align}
where $t_1=\eps t$ is a slow time variable. Observe now that system \eqref{baby-slow} is another two timescale problem, but this time in the nonstandard form \eqref{eq:singpert}. The one-dimensional critical manifold of system \eqref{baby-slow} is $S_1=\{(x_1,x_2)\in \mathbb{R}^2 : x_2=0\}$, a line. 
%
%
Because $$\frac{dx_2}{dt_1} = - a_4 x_2 + \varepsilon a_3\, ,$$ we also  see that every solution to \eqref{baby-slow} converges exponentially (in the slow time $t_1$) to the infra-slow one-dimensional invariant manifold  $S_1^\eps=\{(x_1,x_2)\in \mathbb{R}^2: x_2=\eps \frac{a_3}{a_4}\}$.
The infra-slow dynamics on $S_1^\eps$ is in turn governed by
\begin{align} \label{eq:babyinfra-slow}
\frac{dx_1}{dt_2} &= a_1 - \frac{a_2 a_3}{a_4} x_1\, ,
\end{align}
where $t_2=\eps^2 t$ is an infra-slow time variable. The infra-slow flow on $S_1^\eps$ in fact converges to the equilibrium $x_1=\frac{a_1a_4}{a_2a_3}$ corresponding to the unique equilibrium state $(x_1, x_2, x_3) = (\frac{a_1a_4}{a_2a_3},\eps \frac{a_3}{a_4},\eps)$ of system \eqref{babyreaction}.
Remarkably, and perhaps counter-intuitively, we thus find that system \eqref{babyreaction} actually evolves on {\it three} different timescales.

We would like to point out that our direct computations followed a top-down approach, i.e.~we algorithmically discovered the hidden multiple timescale structure through computing the slow vector field \eqref{baby-slow} in a coordinate chart for the slow manifold $S_0^\varepsilon$, and then identifying \eqref{baby-slow} as another singular perturbation problem. 
%
We note that  one can also identify the three disparate timescales (locally) in this simple example by computing the Jacobian at the unique attracting equilibrium, which is given in upper triangular form
$$
\begin{pmatrix}
-\varepsilon^2 \frac{a_2a_3}{a_4} & -\varepsilon \frac{a_1a_4}{a_3} & \varepsilon a_1\\
0 & -\varepsilon a_4 & \varepsilon a_3 \\
0 & 0 & -1
\end{pmatrix}
\, .
$$
We emphasize that {\em a priori} knowledge (or even existence) of a unique attractor state in general multiple timescale systems is usually exceptional. Hence, any method to identify multiple timescales dynamics should not necessarily rely on this.

\end{ex}
\noindent 
In general, the slow and infra-slow manifolds of a singular perturbation problem cannot be computed as easily as in the above simple example. As a result, it will be hard to predict which, and how many, timescales   a singular perturbation problem will possess.
The parametrisation method presented in this paper will take care of this problem and provide the algorithms needed to define and compute slow and infra-slow manifolds and their flows, in arbitrary multiple timescale problems of the form \eqref{eq:singpert}.

There are already several (families of) algorithms which compute approximations to invariant objects in two-timescale dynamical systems. These include the computational singular perturbation (CSP) method \cite{lam1994, lam1989}, the zero-derivative principle \cite{gear2005}, and the method of intrinsic low-dimensional manifolds \cite{maas1992a,maas1992}. 
Our parametrisation method shares some similarities with the CSP method, which also iteratively computes slow manifolds and fast fibres. A distinction is often drawn between the {\it one-step} and {\it two-step} versions of CSP which approximate, respectively, the slow manifold alone versus the slow manifold and the linear fast fibre bundle simultaneously. Analogously, our parametrisation method is also capable of approximating the slow manifold independently. The CSP method has been implemented for the class of coordinate independent problems of the form \eqref{eq:singpert} in \cite{lizarraga2020}, building on rigorous convergence and coordinate-independence results in \cite{kkz2015,mease1995,valorani2005, kkz2004a,kkz2004b}. We emphasize that our parametrisation method wields the further advantage of being naturally applicable in multiple timescale problems. 


The remainder of this paper is organised as follows.  In section \ref{secmethod}, we introduce the parametrisation method in the   setting of two-timescale problems. We show how the method can be used to compute slow manifolds and their linear fast fibre bundles, and we prove that the method formally converges. In section \ref{secmultiple}, we give a definition of a geometric singular perturbation problem with multiple timescales, and we discuss in some detail the mechanism by which `hidden' timescales can emerge. In section \ref{secappl}, we apply the parametrisation method to two different reaction networks. We conclude with some remarks in section \ref{secremarks}.


 \section{The parametrisation method in two-timescale problems}\label{secmethod}
 
 In this section, we introduce the parametrisation method for finding slow manifolds and their fast fibre bundles in geometric singular perturbation problems of the general form \eqref{eq:singpert}.
 
 \subsection{Setup}
 Let $\varepsilon_0 > 0$, let $U_0\subseteq \mathbb{R}^n$ be an open subset, and consider a smooth parameter-dependent vector field $F: U_0 \times[0,\varepsilon_0) \to \mathbb{R}^n$ of the form\footnote{This could be viewed as either a power series expansion of a given function $F(x,\eps)$ or a finite sum of given functions $F_i(x)$ ($i=1,\ldots,l$).}
$$ 
F(x,\varepsilon)  =\sum_{i=0}^{l} \eps^i F_i(x)\,\qquad (l\ge 1)\,.
$$
Our crucial assumption will be that $F_0: U_0\subseteq\mathbb{R}^n\to \mathbb{R}^n$ admits a smooth $k$-dimensional manifold $S_0$ of critical points, where we only consider the case that $S_0$ is nontrivial, i.e., that $1\leq k \leq n-1$. More specifically, we will assume that $S_0$ is the image of a smooth embedding 
$$\phi_0: U_1\to U_0 \subseteq \mathbb{R}^n\,, $$
where $U_1\subset \mathbb{R}^k$ is an open set. 

Our first goal is to find an invariant manifold $S_0^\eps$ close to $S_0$ for the perturbation $F$ of $F_0$. We do this by searching for an embedding $\phi$ of $S_0^\eps$ with an asymptotic expansion of the form
$$\phi  = \phi(\xi,\varepsilon) = \phi_0(\xi) + \varepsilon  \phi_1 (\xi) + \varepsilon^2  \phi_2 (\xi) + \ldots : U_1 \times [0,\varepsilon_0) \to \mathbb{R}^n\, .$$
It holds that $\phi(U_1,\eps)$ is invariant under the flow of $F(\cdot, \eps)$, for all $\eps \in [0,\eps_0)$, if there is a vector field 
$$
r = r(\xi,\varepsilon) = \varepsilon r_1(\xi) + \varepsilon^2 r_2(\xi) + \ldots : U_1 \times [0,\varepsilon_0) \to \mathbb{R}^k
$$
with the property that whenever $\xi(t)$ satisfies $\frac{d \xi}{dt} = r(\xi,\varepsilon)$, then $x(t):=\phi(\xi(t),\varepsilon)$ satisfies $\frac{d x}{dt} = F(x,\varepsilon)$. 
In other words, if the pair $(\phi, r): U_1 \times [0,\eps_0) \to \mathbb{R}^n \times \mathbb{R}^k$ satisfies the {\it conjugacy equation}
\begin{align}\label{conj}
    \mathcal{F}(\phi, r)(\xi, \eps):= D \phi(\xi,\varepsilon) \cdot r(\xi,\varepsilon) - F(\phi(\xi,\varepsilon),\varepsilon) = 0\, .
\end{align}
Here $D=D_x$ denotes differentiation to the first variable only. In particular, it follows from (\ref{conj}) that $ F(\phi(\xi,\varepsilon),\varepsilon) = D\phi(\xi,\varepsilon)\cdot r(\xi,\varepsilon) \in T_{\phi(\xi,\varepsilon)} \left( \phi(U_1,\eps) \right)$. This means that $F(\phi(\xi),\varepsilon)$ is tangent to $S_0^{\varepsilon}:= \phi(U_1,\eps)$, and hence that  $S_0^{\varepsilon}$ is invariant under the flow of $F(\cdot, \eps)$.
\begin{remk} \label{remk:commute2}
One can equivalently represent  equation (\ref{conj}) by the commuting diagram
$$
\begin{tikzcd}
    U_1 \arrow[r,"\phi "'] \arrow[d,"{\rm id}_{U_1 } \times r "']  &  U_0  \arrow[d,"{\rm id}_{U_0 }\times F "] \\
   U_1 \times \mathbb{R}^k   \arrow[r,"T \phi  "']  & U_0 \times \mathbb{R}^n  
 \end{tikzcd}
$$
where $T\phi(\xi, v,\varepsilon) := (\phi(\xi,\varepsilon), D\phi(\xi,\varepsilon)\cdot v)$. Note that we have suppressed the dependence on the parameter  $\eps\in [0,\eps_0)$ in the diagram for clarity. 
\end{remk}
\noindent 
We will now attempt to solve the conjugacy equation \eqref{conj} by an iterative procedure. In fact, expansion of (\ref{conj}) in powers of $\varepsilon$ yields a recurrence relation for the unknowns $(\phi_i, r_{i})$. The first of these is the equation $F_0\circ \phi_0 = 0$, which holds by definition of $S_0$. The remaining relations are the following (here we don't write the dependence on $\xi$):
%
\begin{align} \label{eq:G-recursive}
\begin{array}{lll}
 D \phi_0  \cdot r_1 - DF_0(\phi_0) \cdot \phi_1   = &  F_1(\phi_0) & =: G_1  \\ 
 D\phi_0 \cdot r_2  - DF_0(\phi_0) \cdot \phi_2   = &  - D\phi_1 \cdot r_1 +  F_2(\phi_0)  &  \\ 
  & + DF_1(\phi_0)\cdot \phi_1  +  \frac{1}{2} D^2F_0(\phi_0)(\phi_1, \phi_1)&  =: G_2 \\ 
  \vdots &   \vdots & \vdots \\ 
 D \phi_0  \cdot r_i - DF_0(\phi_0) \cdot \phi_i   = &  G_i(\phi_0, \ldots, \phi_{i-1}, r_1, \ldots, r_{i-1})  & =: G_i  \\ 
 \vdots & \vdots &  \vdots
\end{array}
\end{align}
%
The $i$-th equation in this list is an inhomogeneous linear equation for $r_i$ and $\phi_i$, of the form
\begin{equation}\label{eq:infconj}
    D\mathcal{F}(\phi_0, 0)\cdot (\phi_i, r_i) = D \phi_0  \cdot r_i - DF_0(\phi_0) \cdot \phi_i= G_i\, . 
\end{equation}
We shall refer to this equation as the {\it infinitesimal conjugacy equation}. 
The inhomogeneous term $G_i=G_i(\xi)$ in this equation only depends on the  $\phi_j$ and $r_j$  with $0\leq j<i$. The recursively defined infinitesimal conjugacy equations  can therefore  be solved iteratively for $(\phi_i, r_i)$ under the condition that the linear operator $D\mathcal{F}(\phi_0, 0)$
is surjective. 

\begin{remk} \label{remk:Nfsplitting}

{In many applications, multiple timescale models are given by a polynomial
vector field $F(x,\varepsilon)$ with the property that the leading order term can be factored as
\begin{align}
F(x,0) = F_0(x)&= {\bf N}_0(x)\cdot f_0(x)\,, \label{eq:splitting}
\end{align}
\noindent 
with matrix ${\bf N}_0: U_0 \to L(\mathbb{R}^{n-k}, \mathbb{R}^{n})$ and vector $f_0: U_0 \to\mathbb{R}^{n-k}$.} 
 If we assume that this matrix ${\bf N}_0(x)$ has full rank $n-k$ for all $x\in S_0$, then $F_0(x)=0$ if and only if $f_0(x)=0$. The critical manifold is then given by $S_0=\{x\in U_0\, | \, f_0(x)=0\}$. This $S_0$ is indeed a manifold of dimension $k$, if it holds that $Df_0(x):\mathbb{R}^n\to\mathbb{R}^{n-k}$ is surjective for all $x\in S_0$. The operator $D\mathcal{F}(\phi_0, 0)$ is given in this case by
 \begin{align}\label{eq:operatorfactor}
D\mathcal{F}(\phi_0, 0)\cdot (\phi_i, r_i) &= \begin{pmatrix} -{\bf N}_0(\phi_0) \cdot Df_0(\phi_0)  & D\phi_0 \end{pmatrix} \begin{pmatrix}  \phi_i \\ r_i \end{pmatrix}.
\end{align}
 \end{remk}

\begin{ex}
Consider the planar system
\begin{align} \label{eq:parab}
\left. \begin{array}{ll}
\begin{pmatrix} x_1' \\ x_2' \end{pmatrix} &= \begin{pmatrix} 2x_1 \\ x_2 \end{pmatrix} (1-x_2-x_1^2) + \eps \begin{pmatrix} 2 \\ -x_1 \end{pmatrix} \vspace{0.3 cm}\\
&= F_0(x_1,x_2) + \eps F_1(x_1,x_2)
\end{array}  \right. 
\end{align}
with $0 \leq \eps \ll 1$. Observe that \eqref{eq:parab} is written with the leading-order part $F_0$ in the factored form \eqref{eq:splitting}. The critical manifold
 $$S_0 = \{(x_1,x_2):F_0(x_1,x_2) = 0\} =  \{(x_1, x_2): \,x_2 = 1- x_1^2\}\, $$
is a parabola, embedded for example by the smooth map
\begin{align}
\phi_0(\xi) &=  \begin{pmatrix}
\xi \\ f_0(\xi)
\end{pmatrix} = \begin{pmatrix}
\xi \\ 1-\xi^2
\end{pmatrix}. \label{eq:parabparam0}
\end{align}
\noindent Recalling that the operator $D\mathcal{F}(\phi_0, 0)$ can also be written in the form \eqref{eq:operatorfactor}, one may compute that it is given by
 \begin{align*}
 D\mathcal{F}(\phi_0, 0) \,\begin{pmatrix}  \phi_{i,1} \\ \phi_{i,2} \\  r_i  \end{pmatrix} (\xi) = 
 \begin{pmatrix}
  4 \xi^2 & 2 \xi  & 1\\ 
  2\xi(1-\xi^2) & 1-\xi^2 & -2 \xi
    \end{pmatrix}\begin{pmatrix} \phi_{i,1}(\xi) \\ \phi_{i,2}(\xi) \\  r_i(\xi)  \end{pmatrix} \, .
\end{align*}
We now make a `graph style' ansatz for the embedding $\phi$ of $S^{\eps}_0$; that is, we take  
$\phi(\xi,\varepsilon) = (\xi,f_0(\xi) + \eps f_1(\xi) + \eps^2 f_2(\xi) + \cdots)$. Equivalently, we choose
\begin{align*}
\phi_i(\xi) &= \begin{pmatrix} \phi_{i,1}(\xi) \\ \phi_{i,2}(\xi) \end{pmatrix} = \begin{pmatrix} 0 \\ f_i(\xi) \end{pmatrix}.
\end{align*}
\noindent It turns out that with this choice, the surjective operator $D\mathcal{F}(\phi_0, 0)$ becomes an invertible operator $\Pi$ acting on $(f_i, r_i)$, because
\begin{align*}
 \begin{pmatrix}
  4 \xi^2 & 2 \xi  & 1\\ 
  2\xi(1-\xi^2) & 1-\xi^2 & -2 \xi
    \end{pmatrix}\begin{pmatrix} 0 \\ f_i(\xi) \\  r_i(\xi)  \end{pmatrix}
= 
\begin{pmatrix}  2 \xi & 1 \\  1 - \xi^2 & -2\xi  \end{pmatrix} \begin{pmatrix} f_i(\xi) \\ r_i(\xi)  \end{pmatrix} =:
\Pi(\xi)  \begin{pmatrix} f_i(\xi) \\ r_i(\xi)  \end{pmatrix}\, ,
\end{align*}
and $\Delta = \Delta(\xi):= - \det \Pi(\xi) = 1+3\xi^2 >0$.
We thus obtain the iterative formula
\begin{align}\label{eq:iterparab}
   \begin{pmatrix} f_i(\xi) \\ r_i(\xi) \end{pmatrix}  &=  \Pi(\xi)^{-1}G_i(\xi) =  \frac{1}{\Delta}\begin{pmatrix} 2\xi & 1 \\ 1-\xi^2 & -2\xi \end{pmatrix}G_i(\xi)\, . 
\end{align}
We proceed to compute the terms $r_1,f_1,r_2,f_2$ explicitly. Using \eqref{eq:iterparab} with $i=1$, we have
\begin{align*}
\begin{pmatrix} f_1(\xi) \\ r_1(\xi) \end{pmatrix} &=   \frac{1}{\Delta}
\begin{pmatrix} 2\xi & 1 \\ 1-\xi^2 & -2\xi \end{pmatrix} F_1(\phi_0(\xi)) \\
&=  \frac{1}{\Delta}\begin{pmatrix} 2\xi & 1 \\ 1-\xi^2 & -2\xi \end{pmatrix} \begin{pmatrix} 2\\ -\xi \end{pmatrix} = \frac{1}{\Delta}\begin{pmatrix} 3\xi \\ 2 \end{pmatrix}.
\end{align*}
We then compute $G_2$:
\begin{align*}
G_2(\xi) &= - D\phi_1(\xi) \cdot r_1(\xi) +  F_2(\phi_0(\xi))  + DF_1(\phi_0(\xi))\cdot \phi_1(\xi)  +  \frac{1}{2} D^2F_0(\phi_0(\xi))(\phi_1(\xi), \phi_1(\xi))\\
&= \begin{pmatrix} 0 \\ (18\xi^2-6)/\Delta^3 \end{pmatrix} + \begin{pmatrix}  0\\0 \end{pmatrix} + \begin{pmatrix} 0 \\ 0 \end{pmatrix} + \begin{pmatrix} 0 \\ -9\xi^2/\Delta^2 \end{pmatrix} = \begin{pmatrix} 0 \\ -3(9\xi^4 - 3\xi^2 + 2)/\Delta^3 \end{pmatrix}.
\end{align*}
So using \eqref{eq:iterparab} for $i = 2$, we find
\begin{align*}
\begin{pmatrix} f_2(\xi) \\ r_2(\xi) \end{pmatrix} &=   \frac{1}{\Delta}\begin{pmatrix} 2\xi & 1 \\ 1-\xi^2 & -2\xi \end{pmatrix}\begin{pmatrix}  0 \\  -3(9\xi^4-3\xi^2+2)/\Delta^3 \end{pmatrix} = \frac{1}{\Delta^4}\begin{pmatrix} -3(9\xi^4-3\xi^2+2) \\  6\xi(9\xi^4 - 3\xi^2+2)  \end{pmatrix}.
\end{align*}
Altogether, this produces the asymptotic expansions
\begin{align*}
r(\xi,\varepsilon) &= \eps \left(\frac{2}{\Delta} \right) + \eps^2 \frac{6\xi(9\xi^4 - 3\xi^2+2)}{\Delta^4}+ \mathcal{O}(\eps^3)\, ,\\
\phi(\xi,\varepsilon) &= \left(\xi, 1-\xi^2 + \eps \left( \frac{3\xi}{\Delta} \right) + \eps^2 \left( \frac{-3(9\xi^4-3\xi^2+2)}{\Delta^4} \right)+ \mathcal{O}(\eps^3)\right)\, .
\end{align*}
\noindent 
We finish this example by pointing out that the embedding $\phi$ of a slow manifold $S_0^\eps$ is not unique; correspondingly, there are many ways to obtain invertible restrictions of the surjective operator $D\mathcal{F}(\phi_0, 0)$. For example, if we let
\begin{align*}
\phi_i(\xi) &= \begin{pmatrix} A(\xi) \\ B(\xi)  \end{pmatrix} g_i(\xi)\, ,
\end{align*}
\noindent 
with $A,B,g_i$ smooth scalar functions, then
\begin{align*}
D\mathcal{F}(\phi_0, 0)\,\begin{pmatrix}  \phi_{i,1} \\ \phi_{i,2}\\ r_i \end{pmatrix} (\xi) &=  \begin{pmatrix} (B(\xi)+2A(\xi)\xi)(2\xi) & 1 \\
 (B(\xi)+2A(\xi)\xi)(1-\xi^2)& - 2\xi \end{pmatrix} \begin{pmatrix} g_i(\xi) \\ r_i(\xi) \end{pmatrix},
\end{align*}
\noindent
and so the operator $D\mathcal{F}(\phi_0, 0)$ restricts to an invertible operator for any nontrivial choice of $A(\xi)$ and $B(\xi)$, except if $B(\xi) = -2\xi A(\xi)$ for some $\xi$. The latter is the case precisely when 
\begin{align}
\begin{pmatrix} A(\xi) \\ B(\xi)  \end{pmatrix}  &= \begin{pmatrix} 1 \\ -2\xi \end{pmatrix}A(\xi) \in \text{im~}D\phi_0(\xi).
\end{align}
The iteration is thus well-defined as long as we seek $\phi_i$ that do not lie in the direction tangent to $S_0$. The choice of $A,B$ in the present example will nevertheless influence the resulting approximations. For example, suppose that we select
\begin{align*}
\phi_1(\xi) &= \begin{pmatrix} A(\xi) \\ B(\xi)  \end{pmatrix} g_1(\xi) = \begin{pmatrix} 2\xi\\ (1-\xi^2) \end{pmatrix} g_1(\xi).
\end{align*}
\noindent 
An identical computation to the above then gives 
\begin{align*}
r_1(\xi) &= \frac{2}{\Delta},\\
\phi_1(\xi) &= \begin{pmatrix} 2\xi \\ (1-\xi^2) \end{pmatrix} \frac{3\xi}{\Delta^2},\\
   r_2 (\xi) &= \frac{6\xi(2+\xi^2+9\xi^4)}{\Delta^4}, \\ &  \text{etc}\, . 
\end{align*}
Observe that these expressions begin to differ from those obtained with the graph style ansatz above.
 \end{ex}

\subsection{The fast fibre bundle and projection operators}
Our assumption that $S_0=\phi_0(U_1)$ consists of critical points of $F_0$ implies that $F_0 \circ \phi_0 = 0$. It follows that $DF_0(\phi_0(\xi)) \cdot D\phi_0(\xi) = 0$ for all $\xi\in U_1$. 
In other words, the tangent space 
$$T_{x}S_0 = {\rm im}\, D\phi_0(\xi)\, \quad  \mbox{where}\ x=\phi_0(\xi)   $$
lies inside the kernel of the Jacobian $DF_0(x)$.  In the following, we assume that this kernel is precisely $k$-dimensional and that $DF_0(x)$ is invertible in the direction transverse to $T_{x}S_0$, $x\in S_0$. More precisely, we make the following definition.

\begin{defi} \label{defnondegenerate} 
Let  $U_1\subset \mathbb{R}^k$  be an open set and let $\phi_0: U_1 \to \mathbb{R}^n$ be an embedding of the critical manifold $S_0$. 
We call $S_0$ {\it normally non-degenerate} if we have:
\begin{itemize}
\item[{\it i)}] a smooth embedding 
$$\Phi_0:U_1 \times \mathbb{R}^{n-k} \to  \mathbb{R}^{n} \times  \mathbb{R}^{n} \ \mbox{of the form}\ \Phi_0(\xi, w) = (\phi_0(\xi), N_0(\xi)\cdot w) \, ,$$
where $N_0: U_1 \to L(\mathbb{R}^{n-k},\mathbb{R}^{n})$ is a family of linear maps, and
\item[{\it ii)}]  a  family of linear invertible maps
$$
n_0: U_1 \to L(\mathbb{R}^{n-k},\mathbb{R}^{n-k})\,,
$$
with the property that 
 \begin{align}\label{linearconjugacy}
 DF_0(\phi_0(\xi)) \cdot N_0(\xi) = N_0(\xi) \cdot n_0(\xi) \ \mbox{for all}\ \xi\in U_1\, .
 \end{align}
 \end{itemize}
 In addition, if the linear maps $n_0(\xi)$, $\xi\in U_1$, have no eigenvalues on the imaginary axis, then $S_0$ is called {\it normally hyperbolic}.
 The corresponding vector bundle 
 $${\cal N}S_0 := \Phi_0(U_1\times \mathbb{R}^{n-k}) \subset \mathbb{R}^n \times \mathbb{R}^n$$
 is called the (linear) {\it fast fibre bundle} of $S_0$. 
\end{defi}


\noindent  The following gives some properties of the fast fibre bundle.
\begin{prop}\label{properties}\mbox{} \\ \vspace{-3mm}
\begin{itemize}
\item[{\it i)}] ${\cal N}S_0$ is a smooth submanifold of $\mathbb{R}^n\times \mathbb{R}^n$.
\item[{\it ii)}] Every linear map $N_0(\xi)$ is injective.
\item[{\it iii)}] $DF_0(\phi_0(\xi))$ leaves ${\rm im}\, N_0(\xi)$ invariant.
\item[{\it iv)}]  The restriction of $DF_0(\phi_0(\xi))$ to ${\rm im}\, N_0(\xi)$ is invertible.
\item[{\it v)}] $N_0(\xi)$ sends each (generalised) eigenvector of $n_0(\xi)$ to a (generalised) eigenvector of $DF_0(\phi_0(\xi))$ with the same eigenvalue.
\item[{\it vi)}] For every $\xi\in U_1$,  ${\rm im}\, N_0(\xi)$ is transverse to $T_{\phi_0(\xi)}S_0 = {\rm im}\, D\phi_0(\xi)$.
\item[{\it vii)}] The maps $n_0(\xi)$ depend smoothly on $\xi$.
\end{itemize}
\end{prop}
\begin{proof}\mbox{} \\ \vspace{-3mm}
\begin{itemize}
\item[{\it i)}]  ${\cal N}S_0=\Phi_0(U_1\times \mathbb{R}^{n-k})$ is the image of a smooth embedding.
\item[{\it ii)}] The assumption that $\Phi_0$ is an embedding implies in particular that every $D\Phi_0(\xi, w)$ is injective. But
$$D\Phi_0(\xi, w) = \left( \begin{array}{cc} D\phi_0(\xi) & 0 \\ \partial_{\xi}N_0(\xi)\cdot w & N_0(\xi)\end{array} \right) \, .$$
So not only $D\phi_0(\xi)$ but also $N_0(\xi)$ is injective for every $\xi$. 

 \item[{\it iii)}]  If $v\in {\rm im}\, N_0(\xi)$, then $v=N_0(\xi)\cdot w$ for some $w$, so $DF_0(\phi_0(\xi)) \cdot v = DF_0(\phi_0(\xi)) \cdot N_0(\xi)\cdot w = N_0(\xi)\cdot n_0(\xi) \cdot w\in {\rm im}\, N_0(\xi)$.
 \item[{\it iv)}] Assume that the restriction of $DF_0(\phi_0(\xi))$ to ${\rm im}\, N_0(\xi)$ is not invertible. It follows that $DF_0(\phi_0(\xi)) \cdot N_0(\xi) = N_0(\xi)\cdot n_0(\xi)$ is not injective. Since $N_0(\xi)$ is injective, this contradicts the injectivity of $n_0(\xi)$.
 \item[{\it v)}]  Suppose $n_0(\xi)w=\lambda w$. Then $DF_0(\phi_0(\xi)) (N_0(\xi) w) = N_0(\xi) n_0(\xi) w = N_0(\xi) \lambda w = \lambda (N_0(\xi)  w)$. For generalised eigenvectors, the argument is similar: if $(n_0(\xi)-\lambda)^k w = 0$ then $(DF_0(\phi_0(\xi)) - \lambda)^k (N_0(\xi) w)= N_0(\xi)  (n_0(\xi) -\lambda)^k w =0$.
 \item[{\it vi)}] Assuming ${\rm im}\, N_0(\xi)$ is not transversal to ${\rm im}\, D\phi_0(\xi)$, these spaces would non-trivially intersect, because $(n-k)+k=n$. This would imply that $DF_0(\phi_0(\xi))\cdot N_0(\xi)=N_0(\xi)\cdot n_0(\xi)$ is not injective because $DF_0(\phi_0(\xi))\cdot D\phi_0(\xi)=0$. This contradicts the injectivity of $n_0(\xi)$. 
  \item[{\it vii)}] Equation (\ref{linearconjugacy}) implies that $$N_0(\xi)^*DF_0(\phi_0(\xi))N_0(\xi) = N_0(\xi)^* N_0(\xi) n_0(\xi)\, ,$$
  where the superscript $^*$ denotes the matrix transpose.
Since $N_0(\xi)$ is injective, $N_0(\xi)^*N_0(\xi)$ is invertible and 
$$n_0(\xi) = (N_0(\xi)^* N_0(\xi))^{-1} N_0(\xi)^*DF_0(\phi_0(\xi))N_0(\xi)\, .$$
This shows that $n_0(\xi)$ is a smooth function of $\xi$.
\end{itemize}
\end{proof}


\begin{remk} 
Our definition that the linear fast fibre bundle ${\cal N}S_0$ is the image of an embedding is more restrictive than the usual one, which only asks for a family of fast subspaces on which $DF_0(\phi_0(\xi))$ is invertible. In this more general case, a global embedding of the fast fibre bundle may not exist. For example, if $S_0$ is a M\"obius strip embedded in $\mathbb{R}^3$ then its normal bundle does not admit a global normal vector. 
\end{remk}

\begin{remk} \label{remk:Nfsplitting2}
In the case that $F_0$ splits as $F_0(x) = {\bf N}_0(x)\cdot { f}_0(x)$ as in Remark \ref{remk:Nfsplitting}, we have
  $T_{x}S_0 = \mbox{ker}\, Df_0(x)$. 
 At the same time,
 $$
 DF_0(x) = {\bf N}_0(x)\cdot Df_0(x)  + D{\bf N}_0(x)\cdot f_0(x)= {\bf N}_0(x)\cdot Df_0(x) \ \mbox{for}\ x\in S_0\,.
 $$
Since $Df_0(x)$ is surjective, it follows that 
$$
{\rm im}\, DF_0(x) = {\rm im}\, {\bf N}_0(x)\,,\qquad \forall x \in S_0\, .
$$
So, under the assumption that ${\bf N}_0(x)$ is transverse to $\ker Df_0(x)$ for all $x\in S_0$, the column vectors of the matrices ${\bf N}_0(x),\, x\in S_0$, span the fast fibre bundle ${\cal N}S_0$. If in addition, a global embedding $\phi_0$ of $S_0$ is given, then we can define $N_0(\xi):={\bf N}_0(\phi_0(\xi))$. The map $\Phi_0(\xi, w) = (\phi_0(\xi), N_0(\xi)\cdot w)$ will then be a global embedding of ${\cal N}S_0$.

In the example system \eqref{eq:parab}, the leading-order term $F_0(x_1,x_2)$ factors as  above, where we may choose for example 
$${\bf N}_0(x_1, x_2) =\begin{pmatrix} 2x_1 \\ x_2 \end{pmatrix}\, .$$
Choosing again $\phi_0(\xi) = (\xi, 1-\xi^2)$ as  global embedding for $S_0$, yields $N_0(\xi)  = \begin{pmatrix} 2\xi \\ 1-\xi^2 \end{pmatrix}\, $
and hence the global embedding of ${\cal N}S_0$ is  given by:
\begin{align*}
\Phi_0(\xi,w) &= (\xi,1-\xi^2,2\xi w, (1-\xi^2)w)\, .
\end{align*}
\end{remk}
\noindent
Under the assumption that the critical manifold $S_0$ is normally non-degenerate (see Definition~\ref{defnondegenerate}), we have the splitting
$$
T_x\mathbb{R}^n=T_x S_0 \oplus {\rm im}\, N_0(\xi) = {\rm im}\, D\phi_0(\xi) \oplus {\rm im}\, N_0(\xi) \,,\qquad \forall\,  x=\phi_0(\xi)\in S_0\,.
$$
This allows us to define a unique family of projection maps $P_0(\xi):\mathbb{R}^n\to \mathbb{R}^n$ (one for each $\xi \in U_1$) where $P_0(\xi)$ projects onto the tangent space $T_xS_0 = {\rm im}\, D\phi_0(\xi)$ along ${\rm im}\, N_0(\xi)$. In the following, we provide a formula for $P_0(\xi)$.
We start by recalling a well-known fact.
\begin{prop}
Let $N:  \mathbb{R}^{n-k}\to \mathbb{R}^n$ be an injective linear map. Then $N^*N$ is invertible and 
$$\pi(N):= (1 - N(N^*N)^{-1}N^*)\, $$ 
is the orthogonal projection onto $\ker N^*$ along ${\rm im}\, N$. 
 \end{prop}
\begin{proof}
First of all, recall that $\ker N^*$ and ${\rm im}\, N$ are each others orthogonal complements. Denote by $\pi(N)$ the projection onto $\ker N^*$ along ${\rm im}\, N$. Then 
$\pi(N)z = z-Ny$ where $y$ is  determined by the requirement that $N^*z-N^*Ny=0$, i.e., $y=-(N^*N)^{-1}N^*z$. This proves that $\pi(N)z=z-N(N^*N)^{-1}N^*z$.
\end{proof}

\begin{lem}\label{projectionlemma}
Let $M:  \mathbb{R}^{k}\to \mathbb{R}^n$ and $N:  \mathbb{R}^{n-k}\to \mathbb{R}^n$ be  transverse linear maps (i.e., their images together span $\mathbb{R}^n$). Then $M^*\pi(N) M$ 
is invertible. 

Denote by $P(M,N): \mathbb{R}^n \to \mathbb{R}^n$ the projection onto ${\rm im} \, M$ along ${\rm im}\, N$, so $P(M,N)$ is the unique linear map satisfying 
$$P(M,N)M=M \ \mbox{and}\ P(M,N)N=0\, .$$ 
Then
\begin{equation}
    P(M,N) = M (M^*\pi(N)M)^{-1}M^*\pi(N)  \, . \label{eq:proj}
\end{equation}
\end{lem}
\begin{proof}
It is not hard to check that $P(M,N)$ satisfy $P(M,N)M=M$ and $P(M,N)N=0$ (note that $\pi(N)N=0$). Therefore $P(M,N)$ is the required projection. It remains to prove that $M^*\pi(N) M$ is invertible, so that the formula provided for $P(M,N)$ is well defined.

To prove this, we first note that $\pi(N)M: \mathbb{R}^k\to \ker N^*$. Now assume that $z\in \ker \, N^*$. Because $M$ and $N$ are transversal, there are (unique) $x$ and $y$ so that $z=Mx+Ny$. But then $0=N^*Mx+N^*Ny$ so $y=-(N^*N)^{-1}N^*Mx$ and thus $z=Mx  - N(N^*N)^{-1}N^*Mx = \pi(N)Mx$.  So $\pi(N)M$ is surjective from $\mathbb{R}^k$ to $\ker N^*$. But $\mathbb{R}^k$ and $\ker N^*$ have equal dimension $k$, so $\pi(N)M$ is injective. Hence, $(\pi(N)M)^*(\pi(N)M)=M^*\pi(N)^*\pi(N)M = M^*\pi(N)^2M=M^*\pi(N)M$ is invertible.
 \end{proof}
 \noindent 
 Applied to $M=D\phi_0(\xi)$ and $N=N_0(\xi)$, Lemma \ref{projectionlemma} gives a formula for the projection map $P_0(\xi)$. 
 It implies in particular  that $P_0(\xi)$ is a smooth function of $\xi$:
 
 \begin{cor}
Let $\phi_0: U_1 \to \mathbb{R}^n$ and $N_0: U_1 \to L(\mathbb{R}^{n-k}, \mathbb{R}^n)$ be smooth functions satisfying the conditions of Definition \ref{defnondegenerate}. For each $\xi \in U_1$, let $P_0(\xi): \mathbb{R}^n\to  \mathbb{R}^n$ denote the projection  onto ${\rm im}\, D\phi_0(\xi)$ along ${\rm im}\, N_0(\xi)$. 

Then $P_0(\xi)$ depends smoothly on $\xi$, that is, the map  $P_0: U_1 \to L(\mathbb{R}^n, \mathbb{R}^n)$ is smooth. 
\end{cor}
 \begin{proof}
 Immediate from Proposition \ref{properties} {\it vi)} and formula (\ref{eq:proj}) in Lemma \ref{projectionlemma}.
 \end{proof}

\begin{remk}
Consider again the parabolic example \eqref{eq:parab}. From \eqref{eq:parabparam0} and Remark \ref{remk:Nfsplitting2} we have 
\begin{align*}
D\phi_0(\xi) &= \begin{pmatrix} 1 \\ -2\xi \end{pmatrix} \ \mbox{and}\ 
N_0(\xi) = \begin{pmatrix}2\xi \\ 1-\xi^2  \end{pmatrix},
\end{align*}
so for each point $x = (\xi,1-\xi^2) \in S_0$, we are furnished with the tangent splitting
\begin{align*}
T_x\mathbb{R}^2 = T_x S_0 \oplus {\rm im}\, N_0(\xi) = {\rm im}\, D\phi_0(\xi) \oplus {\rm im}\, N_0(\xi) 
= \text{im}\,  \begin{pmatrix} 1 \\ -2\xi \end{pmatrix} \oplus  \text{im}\,  \begin{pmatrix} 2\xi \\ 1-\xi^2 \end{pmatrix}.
\end{align*}
Using \eqref{eq:proj}, we find that the corresponding projection $P_0(\xi)$ is given by
\begin{align*}
P_0(\xi) &= D\phi_0(\xi) (D\phi_0(\xi)^* \pi(N_0(\xi)) D\phi_0(\xi))^{-1} D\phi_0(\xi)^* \pi(N_0(\xi))\\
&= \frac{1}{\Delta} \begin{pmatrix} 1-\xi^2 & -2\xi \\ -2\xi (1-\xi^2) & 4\xi^2 \end{pmatrix}\, ,
\end{align*}
where we remind the reader that $\Delta = 1+3\xi^2$. 
\end{remk}

\begin{remk} \label{remk:alternative-proj}
In the case that $F_0$ splits as $F_0(x) = {\bf N}_0(x)\cdot { f}_0(x)$ as in Remark \ref{remk:Nfsplitting} and $S_0$ is normally hyperbolic, we have the alternative formula
$$
P_0(\xi)=1-{\bf N}_0(\phi_0(\xi))(Df_0(\phi_0(\xi)) {\bf N}_0(\phi_0(\xi)))^{-1} Df_0(\phi_0(\xi))
$$
for the projection operator (see, e.g.,\cite{lizarraga2020, wechselberger2020}).
\end{remk}

\subsection{Solving the conjugacy equation}

We are now in a position to present the first main result, which states that the infinitesimal conjugacy equations can all be solved.

\begin{thr}\label{infconjthm}
Consider a singularly perturbed system in its general form \eqref{eq:singpert}
where $S_0$ is normally hyperbolic and the image of a smooth embedding $\phi_0:U_1\to \mathbb{R}^n$.
For any sequence of smooth functions $X_i:U_1\to\mathbb{R}^k$, $i = 1,2,\cdots$, there exist uniquely determined sequences of smooth functions $Y_i: U_1\to \mathbb{R}^{n-k}$, $\phi_i: U_1\to \mathbb{R}^n$, and $r_i:U_1\to \mathbb{R}^k$, which simultaneously satisfy 
\begin{align}\label{phiidef}
\phi_i(\xi) = D\phi_0(\xi)\cdot X_i(\xi) + N_0(\xi)\cdot Y_i(\xi)
\end{align}
and the infinitesimal conjugacy equations
\begin{align} \label{solveinf}
D\phi_0(\xi)  \cdot  r_i(\xi)  - DF_0(\phi_0(\xi)) \cdot  \phi_i(\xi) = G_i(\xi)\, ,
\end{align}  
where the smooth inhomogeneities $G_i: U_1\to \mathbb{R}^n$ are recursively defined in \eqref{eq:G-recursive}.
\end{thr}

\begin{proof}
The ansatz $\phi_i(\xi) = D\phi_0(\xi)\cdot X_i(\xi) + N_0(\xi) \cdot Y_i(\xi)$  transforms the infinitesimal conjugacy equation into 
$$ D\phi_0(\xi) \cdot r_i(\xi)  -   N_0(\xi) \cdot n_0(\xi) \cdot Y_i(\xi) = G_i(\xi)\, .$$
This is true because $DF_0(\phi_0(\xi))\cdot D\phi_0(\xi)=0$ and $DF_0(\phi_0(\xi))\cdot N_0(\xi) = N_0(\xi)\cdot n_0(\xi)$.  
Projecting onto ${\rm im}\, D\phi_0(\xi)$ along ${\rm im}\ N_0(\xi)$ we therefore find 
 $$ D\phi_0(\xi)  \cdot r_i(\xi) = P_0(\xi)(G_i(\xi))\, .$$
 We then apply a left inverse\footnote{In both \eqref{eq:ri} and \eqref{eq:yi}, the left-inverse that we apply is the {\it Moore-Penrose pseudo-inverse}. For any matrix $A$ with full column rank, this pseudo-inverse is defined as $A^+:= (A^* A)^{-1}A^*$. } 
 to obtain
%
\begin{align}  \label{eq:ri}
r_i(\xi) &= (D\phi_0(\xi)^* D\phi_0(\xi))^{-1} D\phi_0(\xi)^* P_0(\xi)(G_i(\xi)) \,.
\end{align}
This formula shows in particular that $r_i$ is unique and depends smoothly on $\xi$.
%
%
Projecting onto ${\rm im}\, N_0(\xi)$ we find on the other hand that 
$$ -N_0(\xi) n_0(\xi) Y_i(\xi) =  (1-P_0(\xi)) (G_i(\xi)) \, .$$
 This implies that 
 \begin{align} \label{eq:yi}
Y_i(\xi) = -n_0(\xi)^{-1} (N_0(\xi)^* N_0(\xi))^{-1} N_0(\xi)^* (1-P_0(\xi)) G_i(\xi) 
\end{align}
is also unique and depends smoothly on $\xi$ as well.
There is no restriction on $X_i(\xi)$.
\end{proof}

\noindent The next corollary states that the conjugacy equation can be solved to arbitrarily high order.
\begin{cor} \label{maincor1}
Let $\phi_0: U_1\to \mathbb{R}^n$ be fixed, and let 
$$X_{(m)} = \varepsilon X_1 + \varepsilon^2 X_2 + \ldots + \varepsilon^m X_m : U_1 \times [0,\eps_0)\to \mathbb{R}^k$$ 
be a smooth function. Then there  are unique smooth functions
\begin{align} \nonumber
& \phi_{(m)} = \phi_0 + \varepsilon \phi_1 + \varepsilon^2 \phi_2 + \ldots + \varepsilon^m \phi_m: U_1 \times [0,\eps_0)\to \mathbb{R}^n\, , \\ \nonumber
& Y_{(m)} = \varepsilon Y_1 + \varepsilon^2 Y_2 + \ldots + \varepsilon^m Y_m: U_1 \times [0,\eps_0)\to \mathbb{R}^{n-k} \, ,\\ \nonumber
& r_{(m)} = \varepsilon r_1 + \varepsilon^2 r_2 + \ldots  + + \varepsilon^m r_m:U_1 \times [0,\eps_0)\to \mathbb{R}^k\, ,
\end{align}
satisfying $\phi_{(m)} = \phi_0 + D\phi_0 \cdot X_{(m)} + N_0 \cdot Y_{(m)}$ and 
$$\mathcal{F}(\phi_{(m)}, r_{(m)})(\xi, \eps) = \mathcal{O}(\varepsilon^{m+1})\, .$$
\end{cor}
 \begin{proof}
 Immediate from Theorem \ref{solveinf}.
  \end{proof}

 \begin{remk} \label{remk:nonunique1}
If $\phi: U_1 \times [0,\varepsilon_0) \to \mathbb{R}^n$ is a smooth family of  embeddings of the  $S_0^\eps$ and
$$h: U_1\times [0, \eps_0) \to U_1 $$
is a parameter family of diffeomorphisms of $U_1$
of the form 
$$(\eta, \eps) \mapsto \xi := h(\eta, \eps)\, , $$
then
 $$\psi = \psi(\eta, \eps) := \phi(h(\eta, \eps), \eps): U_1\times [0,\varepsilon_0)\to \mathbb{R}^n$$ 
 is also an embedding of the 
 $S_0^\eps$. The vector field $r: U_1\times [0,\varepsilon_0)\to \mathbb{R}^k$ for which the conjugacy equation \eqref{conj} holds will then transform. More precisely, if we define the 
  pullback vector field $h^*r:U_1 \times[0,\varepsilon_0)\to \mathbb{R}^k$ by
$$h^*r:= (Dh^{-1} \cdot r) \circ h \, ,$$
then one may  verify that
\begin{align}\label{symmetry}
\mathcal{F}(\phi\circ h, h^*r) = \mathcal{F}(\phi, r) \circ h \, .
\end{align}
This shows that $\mathcal{F}$ possesses a symmetry, its symmetry group being the diffeomorphism group of $U_1$.  In particular, if $\mathcal{F}(\phi, r)=0$ then also $\mathcal{F}(\phi\circ h, h^*r)=0$. Applied to $h={\rm id}_{U_1}+ \varepsilon X$ and 
 expanding  the identity $\mathcal{F}(\phi\circ h, h^*r) = 0$ to first order in $\eps$, then   yields that
$$
D\mathcal{F}(\phi, r) \cdot \left( \begin{array}{c} D\phi \cdot X \\  {\rm [} r, X {\rm ]} \end{array} \right)  = 0\, .
$$
 Here $[r,X] = Dr\cdot X - DX \cdot r$ is the Lie bracket. 
 Applied to $\phi=\phi_0$ and $r=0$ this reduces to 
 $$D\mathcal{F}(\phi_0, 0) \cdot \left( \begin{array}{c} D\phi_0 \cdot X \\  0\end{array} \right)  = 0\, .$$
 This shows very clearly that (and why)  $D\mathcal{F}(\phi_0, 0)$ is not invertible. 
 \end{remk}
 
 \begin{remk}
 On the other hand, Theorem \ref{infconjthm} shows that if we choose $\phi_i$ of the form $\phi_i = D\phi_0\cdot X_i + N_0 \cdot Y_i$, then the solution to the equation $D\mathcal{F}(\phi_0, 0) (\phi_i, r_i) = G_i$ {\it is}  unique after $X_i$ is specified.
 There are certain natural choices for $X_i$ depending on the setting. Suppose for example that the critical manifold admits a parametrisation as a graph of a function, i.e., $\phi_0(\xi) = (\xi, f_0(\xi))$. As we showed in the planar example \ref{eq:parab}, one may then search for a graph style embedding of the slow manifold as well, i.e., of the form
$$\phi_{(m)}(\xi,\varepsilon) = (\xi, f_0(\xi) + \eps f_1(\xi_1) + \cdots + \eps^m f_m(\xi))\, .$$ 
In other words, one requires that $\phi_i(\xi) = (0, f_i(\xi))$. 
For the ansatz $\phi_i = D\phi_0 \cdot X_i + N_0 \cdot Y_i$ from the proof of Theorem \ref{infconjthm} it must then hold that
$$(0, f_i) = \left({\rm id}_{\mathbb{R}^k},  Df_0 \right) \cdot X_i + N_0\cdot Y_i\, .$$
There is thus a simple linear relation between $X_i$ and $Y_i$ of the form $X_i =-p_k(N_0\cdot Y_i)$, where $p_k(x_1, \ldots, x_n)=(x_1, \ldots, x_k)$ is the projection onto the first $k$ coordinates.   

Now recall  from the proof of Theorem \ref{infconjthm} that the solution $(r_i, Y_i)$ of the  $i$-th infinitesimal conjugacy equation is entirely independent of $X_i$. Thus one can first solve for $Y_i$ and choose $X_i:=-p_k(N_0\cdot Y_i)$ afterwards, yielding the desired function $f_i$ for which $\phi_i=(0,f_i)$. 
\end{remk}

\subsection{Computation of the perturbed fast fibre bundle} \label{fibresec}

Here we show that the parametrisation method can also be used to compute an expansion of the perturbed linear fast fibre bundle ${\cal N}S_0^\eps$ of ${\cal N}S_0$. To this end, we augment the ODE $\frac{dx}{dt}=F(x,\varepsilon)$ with its variational equations, i.e., we study the system of ODEs
\begin{align}\label{eq:extended}
\begin{pmatrix}
\frac{dx}{dt}\\
\frac{dv}{dt}
\end{pmatrix}
=
\begin{pmatrix}
F(x,\varepsilon)\\
DF(x,\varepsilon)\cdot v
\end{pmatrix}
=
\begin{pmatrix}
F_0(x) + \varepsilon F_1(x) + \ldots\\
(DF_0(x) + \varepsilon DF_1(x) + \ldots )\cdot v
\end{pmatrix}
=: TF(x,v,\varepsilon)\, .
\end{align}
We first make the following simple observation.
\begin{prop}
The embedding 
$\Phi_0: U_1\times \mathbb{R}^{n-k} \to \mathbb{R}^n \times \mathbb{R}^n$ of ${\cal N}S_0$ defined by 
$$\Phi_0(\xi, w) = (\phi_0(\xi), N_0(\xi)\cdot w)$$
 conjugates the vector field $R_0: U_1\times \mathbb{R}^{n-k} \to \mathbb{R}^k \times \mathbb{R}^{n-k}$  defined by 
$$R_0(\xi, w) := (0, n_0(\xi)\cdot w)$$
to the variational vector field $TF_0$, i.e., 
$$D\Phi_0(\xi, w)\cdot R_0(\xi, w) = TF_0(\Phi_0(\xi, w)) \, .$$ 
\end{prop}
\begin{proof}
\begin{align} \nonumber 
& D\Phi_0(\xi, w)\cdot R_0(\xi, w) = \left( \begin{array}{cc} D\phi_0(\xi) & 0 \\ \partial_{\xi}N_0(\xi)w & N_0(\xi)\end{array} \right)\cdot \left( \begin{array}{c} 0 \\ n_0(\xi)\cdot w\end{array} \right) = \\ \nonumber
& \left( \begin{array}{c} 0 \\ N_0(\xi) \cdot n_0(\xi)\cdot w\end{array} \right)  = \left( \begin{array}{c} F_0(\phi_0(\xi)) \\ DF_0(\phi_0(\xi)) \cdot N_0(\xi) \cdot w\end{array} \right) = TF_0(\Phi_0(\xi, w))\, .
\end{align}
Note that the penultimate equality follows from equation (\ref{linearconjugacy}).
\end{proof}
\noindent As an immediate consequence, we see that ${\cal N}S_0$ is an invariant manifold for the flow of the variational vector field $TF_0$.
The goal of this section is to find a nearby invariant manifold ${\cal N}S_0^\eps$ for the flow of $TF$ by searching for a parametrisation 
$$\Phi: (\xi, w,\varepsilon) \mapsto (\phi(\xi,\varepsilon), N(\xi,\varepsilon) \cdot w) \ \mbox{from}\ U_1 \times \mathbb{R}^{n-k} \times [0,\varepsilon_0) \ \mbox{to}\ \mathbb{R}^n \times \mathbb{R}^n$$
of the form
$$\Phi(\xi, w,\varepsilon) = (\phi_0(\xi) + \varepsilon \phi_1(\xi) + \ldots, (N_0(\xi)+ \varepsilon N_1(\xi) + \ldots)w)$$
and an accompanying reduced vector field 
$$R: (\xi, w,\varepsilon) \mapsto (r(\xi,\varepsilon), n(\xi,\varepsilon)\cdot w) \ \mbox{from}\ U_1 \times \mathbb{R}^{n-k} \times [0,\varepsilon_0) \ \mbox{to}\ \mathbb{R}^k \times \mathbb{R}^{n-k}\, , $$
of the form
$$R(\xi, w,\varepsilon) = (  \varepsilon r_1(\xi) + \varepsilon^2 r_2(\xi) + \ldots, (n_0(\xi)+ \varepsilon n_1(\xi) + \ldots)w)\, .$$
We want $\Phi$ to send integral curves of $R$ to integral curves of $TF$, so we impose the conjugacy equation
$D\Phi(\xi, w,\eps)\cdot R(\xi, w,\eps) = TF(\Phi(\xi, w,\eps),\varepsilon)$, where $D=D_{(x,w)}$, or equivalently,
\begin{align}\nonumber
\mathcal{G}(\phi, r, N, n)(\xi, \eps) := \left(
\begin{array}{l}
 D\phi(\xi,\varepsilon)\cdot r(\xi,\varepsilon) - F(\phi(\xi,\varepsilon),\varepsilon)  \\ 
  N(\xi,\varepsilon) \cdot n(\xi,\varepsilon) +  \partial_{\xi} N(\xi,\varepsilon) \cdot r(\xi,\varepsilon) - DF(\phi(\xi,\varepsilon),\varepsilon) \cdot N(\xi,\varepsilon) 
\end{array} 
\right) = 0\, .
\end{align}
Note that the first component of $\mathcal{G}$ is simply $\mathcal{F}$, and we have already considered the equation ${\cal F}(\phi, r) = 0$ in the previous sections. We proceed to expand the second component of $\mathcal{G}$ in powers of $\varepsilon$. Collecting terms of the same order, and not writing $\xi$ anymore, we obtain the following recurrent list of equations.
\begin{align}\label{eq:H-recursive} 
\begin{array}{lll}
N_0 \cdot n_0 - DF_0(\phi_0) \cdot N_0  &   & = 0
\\ 
 N_1 \cdot n_0 + N_0 \cdot n_1 - DF_0(\phi_0) \cdot N_1 = &  - \partial_{\xi} N_0 \cdot r_1 + DF_1(\phi_0) \cdot N_0 & \\
 & + \frac{1}{2}D^2F_0(\phi_0)(\phi_1)\cdot N_0 & =: H_1 \\
 \vdots & \vdots & \vdots \\
  N_i \cdot n_0 + N_0\cdot n_i - DF_0(\phi_0) \cdot N_i = &  H_i(\phi_0, \ldots, N_{i-1} ) & =: H_i \\
 \vdots & \vdots & \vdots 
 \end{array}
\end{align}
The first equation is satisfied by assumption, see (\ref{linearconjugacy}). The right-hand sides $H_i=H_i(\xi)$ only depend on the $\phi_j, r_j$ with $0\leq j \leq i$ and the $n_j, N_j$ and $\partial_{\xi} N_j$ with $0\leq j <i$, so we think of it as an inhomogeneous term.
We can iteratively solve these equations if  the operator
$$(N_i, n_i) \mapsto N_i n_0+N_0 n_i  - DF_0(\phi_0) \cdot N_i   \, .$$
is surjective. Again, this operator is far from  injective, so the solutions to the recurrence relations are not unique.  Nevertheless we have the following result, which states that the `variational infinitesimal conjugacy equations' can all be solved.
\begin{thr}\label{solveinflinear}
Consider the augmented system \eqref{eq:extended} under the assumption that the critical manifold $S_0$ is normally hyperbolic.
For any sequence of smooth $M_i: U_1\to  L(\mathbb{R}^{n-k}, \mathbb{R}^{n-k})$, there exist unique sequences of smooth $N_i: U_1\to   L(\mathbb{R}^{n-k}, \mathbb{R}^n)$, $L_i: U_1\to   L(\mathbb{R}^{n-k}, \mathbb{R}^k)$,  and $n_i: U_1\to  L(\mathbb{R}^{n-k}, \mathbb{R}^{n-k})$  which simultaneously satisfy 
$$
N_i(\xi) = N_0(\xi)\cdot M_i(\xi) + D\phi_0(\xi) \cdot L_i(\xi)\, , $$
as well as the infinitesimal variational conjugacy equation
\begin{equation}\label{eq:Hrecursive}
N_i (\xi) \cdot n_0(\xi) + N_0(\xi)\cdot n_i(\xi)  - DF_0(\phi_0(\xi)) \cdot N_i(\xi) = H_i(\xi)\,,
\end{equation}
where the smooth inhomogeneities $H_i: U_1\to L(\mathbb{R}^{n-k}, \mathbb{R}^n)$ are recursively defined in \eqref{eq:H-recursive}.
\end{thr} 

\begin{proof}
Imposing the special form of $N_i(\xi)$ transforms the $i$-th infinitesimal variational conjugacy equation (\ref{eq:Hrecursive}) into
$$ D\phi_0 \cdot L_i \cdot n_0 + N_0 \cdot n_i  + \underbrace{N_0 \cdot M_i \cdot n_0 - N_0  \cdot n_0 \cdot M_i}_{N_0\cdot [M_i, n_0]} = H_i  \, .$$
Here we did not write $\xi$ and we used that  $DF_0(\phi_0)\cdot D\phi_0 = 0$ and that $DF_0 \cdot N_0 = N_0 \cdot n_0$.
Projecting onto ${\rm im}\, D\phi_0(\xi)$, we obtain $$D\phi_0(\xi) \cdot  L_i(\xi) \cdot  n_0(\xi)  = P_0(\xi)(H_i(\xi)) \, .$$ 
So we find that 
\begin{align}\label{eq:li}
L_i(\xi) &= (D\phi_0(\xi)^* D\phi_0(\xi))^{-1} \cdot D\phi_0(\xi)^* \cdot P_0(\xi)(H_i(\xi)) \cdot n_0(\xi)^{-1} \,
\end{align}
is unique and depends smoothly on $\xi$. 
Similarly, projecting onto ${\rm im}\, N_0(\xi)$, we obtain 
$$ N_0(\xi) \cdot n_i(\xi)    = (1-P_0(\xi))(H_i(\xi)) -  N_0\cdot [ M_i(\xi),  n_0(\xi) ] ) \, .$$  
So also
\begin{align}\label{eq:ni}
n_i(\xi) &= ( N_0(\xi)^*N_0(\xi) )^{-1} N_0(\xi)^* \left((1-P_0(\xi))(H_i(\xi)) -  N_0(\xi)\cdot [ M_i(\xi),  n_0(\xi) ]  \right) \, .
\end{align}
is unique and a smooth function of $\xi$.
\end{proof}

\begin{remk}
The embedding $\Phi$ of the fast fibre bundle is not unique. In fact, we can reparameterise it using any parameter family of diffeomorphisms 
$$H: U_1\times \mathbb{R}^{n-k}\times [0, \eps_0)  \to U_1\times \mathbb{R}^{n-k}$$
of the form 
$$
(\eta, v, \eps) \mapsto (\xi, w) := (h(\eta, \eps), M(\eta, \eps)\cdot v)\,,
$$
with $h$ any smooth family of diffeomorphisms of $U_1$ and $M$ any smooth family of invertible transformations of $\mathbb{R}^{n-k}$. This results in another embedding of ${\cal N}S_0^{\eps}$, given by
$$\Phi\circ H: (\eta, v,\eps)\mapsto (\phi(h(\eta, \eps),\eps),  N(h(\eta, \eps),\eps)\cdot M(\eta, \eps) \cdot v)\,. $$
Analogous to the discussion in Remark \ref{remk:nonunique1}, the terms in the expansion of the embedding $\Phi$ of the perturbed fast fibre bundle become unique once the sequence $\{M_i\}$ is chosen. The choice $M_i = 0$ is typically convenient in practical computations.
\end{remk}
\noindent
The following corollary states that the variational conjugacy equation can be solved to arbitrarily high order.

\begin{cor} 
Let $\phi_0: U_1\to \mathbb{R}^n$ be fixed, and let 
\begin{align} \nonumber 
& X_{(m)} = \varepsilon X_1 + \varepsilon^2 X_2 + \ldots + \varepsilon^m X_m : U_1 \times [0,\eps_0) \to \mathbb{R}^k \\ \nonumber
& M_{(m)} = \varepsilon M_1 + \varepsilon^2 M_2 + \ldots + \varepsilon^m M_m: U_1 \times [0,\eps_0)\to   L(\mathbb{R}^{n-k}, \mathbb{R}^{n-k})
\end{align}
be smooth functions. Then there  are unique smooth functions
\begin{align} \nonumber
& \phi_{(m)} = \phi_0 + \varepsilon \phi_1 + \varepsilon^2 \phi_2 + \ldots + \varepsilon^m \phi_m: U_1 \times [0,\eps_0)\to \mathbb{R}^n\, , \\ \nonumber
& Y_{(m)} = \varepsilon Y_1 + \varepsilon^2 Y_2 + \ldots + \varepsilon^m Y_m: U_1 \times [0,\eps_0)\to \mathbb{R}^{n-k} \, ,\\ \nonumber
& r_{(m)} = \varepsilon r_1 + \varepsilon^2 r_2 + \ldots  + + \varepsilon^m r_m:U_1 \times [0,\eps_0)\to \mathbb{R}^k\,  , \\ \nonumber
& N_{(m)} = N_0 + \varepsilon N_1 +  \varepsilon^2 N_2+  \ldots + \varepsilon^m N_m: U_1 \times [0,\eps_0)\to L(\mathbb{R}^{n-k}, \mathbb{R}^{n}) \, ,   \\ \nonumber
& L_{(m)} = \varepsilon L_1 +  \varepsilon^2 L_2+  \ldots + \varepsilon^m L_m : U_1 \times [0,\eps_0)\to L(\mathbb{R}^{n-k}, \mathbb{R}^{k})\, , \\ \nonumber
& n_{(m)} =  n_0 + \varepsilon n_1 +  \varepsilon^2 n_2+  \ldots + \varepsilon^m n_m :    U_1 \times [0,\eps_0)\to L(\mathbb{R}^{n-k}, \mathbb{R}^{n-k}) \, ,
\end{align}
satisfying 
$$
\phi_{(m)} = \phi_0 + D\phi_0 \cdot X_{(m)} + N_0 \cdot Y_{(m)} \ \mbox{and}\  N_{(m)} = N_0 + N_0\cdot M_{(m)} + D\phi_0 \cdot L_{(m)}$$ 
and
$$\mathcal{G}(\phi_{(m)}, r_{(m)}, N_{(m)}, n_{(m)})(\xi, \eps) = \mathcal{O}(\varepsilon^{m+1})\, .$$
\end{cor}
   \begin{proof}
 Immediate from Theorem \ref{solveinflinear}.
  \end{proof}

\begin{ex}
We compute the first-order corrections to the fast fibre bundle and the flow on the fibres for example \eqref{eq:parab}. We first note that

\begin{align*}
DF_0(\phi_0(\xi)) N_0(\xi) = \begin{pmatrix} -4\xi^2 & -2\xi \\ 2\xi(\xi^2-1) &\xi^2 -1  \end{pmatrix} \begin{pmatrix} 2\xi \\  1-\xi^2 \end{pmatrix} 
 = - \begin{pmatrix} 2\xi \\ 1- \xi^2  \end{pmatrix} (1+3\xi^2) = - N_0(\xi) \Delta(\xi)\, ,
\end{align*}
and so we see that 
$$n_0(\xi) = -\Delta(\xi)\, .$$
Using \eqref{eq:H-recursive} we compute
\begin{align*}
H_1(\xi) &= - \partial_{\xi} N_0(\xi)r_1(\xi) + DF_1(\phi_0(\xi)) \cdot N_0(\xi) + \frac{1}{2}D^2F_0(\phi_0(\xi))(\phi_1(\xi),N_0(\xi))\\
&= \begin{pmatrix} \xi^2 -3 - \frac{2}{\Delta(\xi)} \\  \xi\left( \frac{2}{\Delta(\xi)} + \xi -1 \right) - 1 \end{pmatrix}.
\end{align*}
It then follows from \eqref{eq:li} that
\begin{align*}
L_1(\xi) &= (D\phi_0(\xi)^* D\phi_0(\xi))^{-1} \cdot D\phi_0(\xi)^* \cdot P(\xi)(H_1(\xi)) \cdot n_0(\xi)^{-1} \\
&= \frac{5-2\xi + 5\xi^2 - 4\xi^3 - 17 \xi^4 + 6\xi^5 +3\xi^6}{\Delta(\xi)^3}.
\end{align*}
The $\mathcal{O}(\eps)$ term in the expansion of the  fast fibre bundle is given by
$
N_1(\xi) = N_0(\xi)\cdot M_1(\xi) + D\phi_0(\xi) \cdot L_1(\xi)
$, where 
 $M_1(\xi)$ may be chosen freely. Let us choose  $M_1(\xi) \equiv 0$. Using \eqref{eq:ni}, we then have
\begin{align*}
n_1(\xi) &= ( N_0(\xi)^*N_0(\xi) )^{-1} N_0(\xi)^* ((1-P_0(\xi))(H_1(\xi)) \\
&= \frac{-1-9\xi-2\xi^2-19\xi^3+3\xi^4+6\xi^5}{\Delta(\xi)^2}
\end{align*}
Altogether, this shows that the dynamics on the linear fast fibre bundle (i.e., of the original dynamical system augmented by its variational equations) is conjugate to
 \begin{align*}
 \frac{d\xi}{dt} & = r(\xi, \eps) =   \eps \left(\frac{2}{\Delta(\xi)} \right) + \eps^2 \left( \frac{6\xi(9\xi^2 + 9\xi^4-2)}{\Delta(\xi)^4}\right) + \cdots \\ 
 \frac{dw}{dt} & = n(\xi, \eps) w = \left(  -\Delta(\xi) + \eps \left( \frac{-1-9\xi-2\xi^2-19\xi^3+3\xi^4+6\xi^5}{\Delta(\xi)^2}  \right) +\cdots \right)w.\, 
 \end{align*}
This concludes our example.
\end{ex}

\section{Multiple timescale problems} \label{secmultiple}
Inspired by the geometric framework of the parametrisation method, we now introduce a new,   geometric characterisation of a multiple timescale system, using a sequence of nested embedded slow manifolds. This new characterisation can be viewed as an extension of the standard form \eqref{eq:cardin1}; in particular, it is broad enough to include systems with `hidden' dynamical timescales, such as  example \eqref{babyreaction} in the introduction. 

 \begin{defi} \label{def:mult}
Let $\eps_0>0$ and $U_0 \subset \mathbb{R}^n$ a nonempty open subset. Consider a smooth one-parameter family of ODEs $x' = F(x,\eps)$ with $F: U_0 \times [0,\eps_0) \to \mathbb{R}^n$.  Let $m \geq 2$ be an integer. 

\indent We say that $F$ defines an {\it $m$-timescale dynamical system} if there exist integer sequences $n = n_0 > n_1 > n_2 > \ldots > n_{m-1}>0$ and $0=j_0 < j_1 < j_2 < \ldots < j_{m-1}$, such that 
\begin{itemize}
    \item[{\it i)}] for every $0\leq i \leq m-1$ there is a smooth one-parameter family of vector fields $$r^{(i)}: U_i\times [0,\eps_0) \to \mathbb{R}^{n_i}$$ defined on a nonempty open subset $U_i\subset \mathbb{R}^{n_i}$, and admitting an asymptotic expansion 
$$r^{(i)}(\xi, \eps) = \eps^{j_i}r^{(i)}_{j_i}(\xi) + \mathcal{O}(\varepsilon^{j_i+1}) \, ;$$
    We have $r^{(0)} = F$, and thus $r^{(0)}_0 = F_0$;
\item[{\it ii)}] 
for every $0\leq i \leq m-1$ the critical manifold 
$$S_i := \{ \xi \in U_{i} \, |\, r^{(i)}_{j_i}(\xi)=0 \} \subset U_i$$ 
of $r^{(i)}$ has dimension $n_{i+1}$ and is normally hyperbolic (with respect to  $Dr^{(i)}_{j_i}$, see Def. \ref{defnondegenerate});

    \item[{\it iii)}] for every $0\leq i \leq m-2$, there is  a smooth one-parameter family of embeddings $$\phi^{(i)}: U_{i+1}\times [0,\varepsilon_0) \to U_i \, $$
    admitting an asymptotic expansion
    $$\phi^{(i)}(\xi, \eps) = \phi^{(i)}_0(\xi) + \mathcal{O}(\eps)\, .$$
It holds that
$$ S_i = \phi^{(i)}_0(U_{i+1})\, ;$$
\item[{\it iv)}] for each $0\leq i \leq m-2$ and each $0\leq \eps < \eps_0$ the conjugacy equation
$$D\phi^{(i)}\cdot  r^{(i+1)} = r^{(i)}\circ \phi^{(i)}$$
holds. Here, $D = D_{\xi}$ denotes the derivative with respect to the first variable. We write $$S_i^{\eps} := \phi^{(i)}(U_{i+1}, \eps) \subset U_i$$ for the corresponding (infra-)slow manifold.
\end{itemize}
 \end{defi}
 \noindent The $n_{i}$-dimensional manifold $S_{i-1}^{\eps} :=\phi^{(i)}(U_i, \eps)$ can be interpreted as the $i$-th slow manifold of the multiple timescale system.  Note that we have introduced superscripts ${}^{(i)}$ to denote objects related to $S_{i-1}^{\eps}$. The integers $n_i$
   decrease, expressing that the slow, infra-slow, etc. manifolds decrease in dimension. The flow on $S^{\eps}_{i-1}$ is governed by the vector field $r^{(i)}$ on $U_{i}$, and takes place on a timescale $t_i = \varepsilon^{j_i} t$. The integers $j_i$ increase, expressing that the timescales become slower as we descend along the chain of nested slow manifolds. 
\begin{remk}
The dimension $n_{i}$ of $S_{i-1}$ is generally unrelated to the leading order $j_i$ of the vector field $r^{(i)}$ on it. This explains the need for two separate lists of integers $n_i$ and $j_i$. 
\end{remk}

\begin{remk} \label{remk:commuten}

In analogy to Remark \ref{remk:commute2}, the conjugacy relations in Def. \ref{def:mult} can be expressed compactly by a commutative diagram:
$$
 \begin{tikzcd}
U_{m-1} \arrow[r,"\phi^{(m-2)}"] \arrow[d,"\text{id}_{U_{m-1}}\times r^{(m-1)}"]  &  U_{m-2} \arrow[r," "] \arrow[d,"\text{id}_{U_{m-2}}\times r^{(m-2)}"]  &  ~\cdots ~  \arrow[r," "] \arrow[d," "]  &  U_1 \arrow[r,"\phi^{(0)}"] \arrow[d,"\text{id}_{U_1}\times r^{(1)}"]  &  U_0 \arrow[d,"\text{id}_{U_0} \times r^{(0)}"]  \\
U_{m-1}\times \mathbb{R}^{n_{m-1}} \arrow[r,"T\phi^{(m-2)}"']  &  U_{m-2} \times \mathbb{R}^{n_{m-2}} \arrow[r," "'] & ~\cdots~  \arrow[r," "']  & U_1\times \mathbb{R}^{n_1} \arrow[r,"T\phi^{(0)}"']  &  U_0 \times \mathbb{R}^{n_0}
 \end{tikzcd}
$$
Each square depicts a conjugacy linking the dynamics on an invariant submanifold to dynamics on a corresponding chart. Moving from right to left in the diagram, we descend the chain of fast, slow, infra-slow, etc. manifolds.
\end{remk}

 

\subsection{Hidden timescales} \label{sec:hidden}

 The surprising appearance of {\it hidden} timescales, as observed in the introductory example \eqref{babyreaction}, can be explained by the presence of terms other than $F_i$ in the inhomogeneous term $G_i$ of the infinitesimal conjugacy equations \eqref{eq:G-recursive}. In this section we will explore this phenomenon for systems with three timescales. 
 
 Here, we assume that $F(x, \eps) = F_0(x)+ \eps F_1(x) + \eps^2 F_2(x) + \mathcal{O}(\eps^3)$ is a three-timescale vector field on $U_0\subset\mathbb{R}^n$, with a critical manifold $S_0 = \{x\in U_0\, |\, F_0(x)=0\} = \phi_0^{(0)}(U_1)\subset \mathbb{R}^n$, reduced slow vector field 
 $$r^{(1)}(\xi, \eps) = \eps r^{(1)}_1(\xi) + \eps^2 r^{(1)}_2(\xi) + \mathcal{O}(\eps^3)\ \mbox{on}\ U_1\subset \mathbb{R}^{n_1}\, ,$$ 
 infra-critical manifold $S_1 = \{\xi \in U_1 \, |\,  r_1^{(1)}(\xi) = 0\} = \phi_0^{(1)}(U_2)$ and  reduced infra-slow vector field 
 $$
 r^{(2)}(\eta, \eps) = \eps^2 r^{(2)}_2(\eta) + \mathcal{O}(\eps^3) \ \mbox{on}\ U_2 \subset \mathbb{R}^{n_2}\, .
 $$
 
 \begin{prop}
 The leading order term $r_2^{(2)}: U_2\to \mathbb{R}^{n_{2}}$   is given by
 \begin{align}\label{formular22}
 r_2^{(2)}(\eta) = \tilde P_{0}^{(1)}(\eta) \tilde P_{0}^{(0)}(\xi) G_2(\xi) \ \mbox{with}\ \xi=\phi_0^{(1)}(\eta)\, . 
 \end{align}
 Here, $G_2(\xi)$ is as defined in \eqref{eq:G-recursive}, and the 
    {\it auxiliary projections} $\tilde{P}_0^{(0,1)}$ are defined by
\begin{equation}\label{eq:tildeP}
\begin{aligned} 
\tilde{P}_0^{(0)}(\xi) &= (D\phi^{(0)}_0(\xi)^* D\phi_0^{(0)}(\xi))^{-1} D\phi^{(0)}_0(\xi)^* P_0^{(0)}(\xi)\, , \\
\tilde{P}_0^{(1)}(\eta) &= (D\phi^{(1)}_0(\eta)^* D\phi_0^{(1)}(\eta))^{-1} D\phi^{(1)}_0(\eta)^* P_0^{(1)}(\eta)\, ,
\end{aligned}
\end{equation}
where $P_0^{(0)}$ denotes the projection onto the tangent space of $S_0$ along the fast fibre bundle of $DF_0$ and $P_0^{(1)}$ denotes the projection onto the tangent space of $S_1$ along the fast fibre bundle of $Dr_1^{(1)}$.  
\end{prop}
 
 \begin{proof}
Projection of the second equation in \eqref{eq:G-recursive} gives  
 \begin{align}\nonumber 
 & D\phi^{(0)}_0(\xi)\cdot r_2^{(1)}(\xi) = P_0^{(0)}(\xi)\cdot G_2(\xi)\, . 
\end{align}
Formula \eqref{eq:ri} shows that the solution to this equation reads
\begin{align}
\label{first}
 r_2^{(1)}(\xi) = \tilde P_0^{(0)}(\xi) G_2(\xi)\, .
\end{align}
Similarly, the first equation in \eqref{eq:G-recursive} with $F$ replaced by $\eps^{-1}r^{(1)}$, $r^{(1)}$ by $\eps^{-1} r^{(2)}$ and $\phi$ by $\phi^{(1)}$ implies that 
 \begin{align}\nonumber
  D\phi^{1}_0(\eta)\cdot r_2^{(2)}(\eta) = P_0^{(1)}(\eta)\cdot G_1(\eta)\, ,
 \end{align}
 with $G_1(\eta) = r_1^{(2)}(\phi_0^{(1)}(\eta))$. The solution to this equation reads
 \begin{align}
     \label{second}
   r_2^{(2)}(\eta) = \tilde P_0^{(1)}(\eta) G_1(\eta) =  \tilde P_0^{(1)}(\eta) r_1^{(2)}(\xi)\ \mbox{with} \ \xi = \phi_0^{(1)}(\eta)\, .
 \end{align}
 Together, \eqref{first} and \eqref{second} give formula \eqref{formular22}.
 \end{proof}
 \noindent It was shown in \eqref{eq:G-recursive} that the inhomogeneous term $G_2  = G_2(\xi)$ is given by the sum
\begin{equation}\label{eq:G2}
G_2  = -D\phi^{(0)}_1 r^{(1)}_1 + F_2(\phi^{(0)}_0) +DF_1(\phi^{(0)}_0)\cdot\phi^{(0)}_1 + \frac{1}{2}D^2 F_0(\phi^{(0)}_0)(\phi^{(0)}_1,\phi^{(0)}_1)\, .
\end{equation}
Note that the term $-D\phi^{(0)}_1 r^{(1)}_1$ does not contribute to the expression for $r_2^{(2)}$ in formula \eqref{formular22} because $r^{(1)}_1(\xi) = 0$ for $\xi = \phi_0^{(1)}(\eta) \in S_1$ by definition of $S_1$. The contribution of the term $F_2(\phi^{(0)}_0)$ is expected; a nontrivial contribution to $r_2^{(2)}$ from this term is not ``hidden.'' The two remaining terms in \eqref{eq:G2} are more surprising. It turns out that both can contribute nontrivially to $r_2^{(2)}$, i.e. to an infra-slow flow on the timescale $t_2=\eps^2 t$.  We shall demonstrate this using two simple examples of the form ${x}' = F_0(x) + \eps F_1(x)$ (i.e. setting $F_2 \equiv 0$).\\

\begin{ex}[{\it Nontrivial contribution of $DF_1(\phi^{(0)}_0)\cdot\phi^{(0)}_1$}]
We revisit the IFFL model \eqref{babyreaction} in the introduction, this time applying the parametrisation method to pinpoint which term(s) in \eqref{eq:G2} contribute to the infra-slow flow. The critical manifold
\begin{align*}
S_0 &= \{(x_1,x_2,x_3): x_3 = 0\}
\end{align*}
admits a smooth parametrisation $\phi^{(0)}_0: U_1:= \mathbb{R}^2 \to U_0 :=\mathbb{R}^3$ defined by 
\begin{align*}
\phi^{(0)}_0(\xi_1,\xi_2) &= \begin{pmatrix} \xi_1 \\ \xi_2 \\ 0 \end{pmatrix}.
\end{align*}
We also record the factorisation of the leading-order part of the vector field, of the form \eqref{eq:splitting}:
\begin{align*}
F_0(x_1,x_2,x_3) &= \begin{pmatrix} 0 \\ 0 \\ -x_{3} \end{pmatrix} = \begin{pmatrix} 0 \\ 0 \\ -1 \end{pmatrix} x_3,
\end{align*}
implying that
\begin{align*}
N_0^{(0)}(\xi_1,\xi_2) &= \begin{pmatrix} 0 \\ 0 \\ -1 \end{pmatrix}
\end{align*}
spans the fast fibres with basepoints in $S_0$ (see Remark \ref{remk:Nfsplitting2}).

Let us proceed to compute $r^{(1)}_1$ and $\phi^{(0)}_1$ using the parametrisation method. For simplicity we choose the graph ansatz
\begin{align*}
\phi^{(0)}(\xi_1,\xi_2, \eps) &= \begin{pmatrix} \xi_1 \\ \xi_2 \\ 0 \end{pmatrix} + \eps \begin{pmatrix} 0 \\ 0 \\ f_1(\xi_1,\xi_2)  \end{pmatrix} + \eps^2 \begin{pmatrix} 0 \\ 0 \\ f_2(\xi_1,\xi_2) \end{pmatrix} + \cdots.
\end{align*}
Writing $r^{(1)}_1 = (r^{(1)}_{1,1}, r^{(1)}_{1,2})$ and realising that $\phi^{(0)}_1 = (0, 0, f_1)$, the first infinitesimal conjugacy equation 
$D\phi^{(0)}_0 \cdot r^{(1)}_1 - DF_0(\phi_0^{(0)})\cdot \phi_1^{(0)}= F_1(\phi_0)$ reduces to the equation
\begin{align*}
\begin{pmatrix}
1 & 0 \\ 0 & 1 \\ 0 & 0 
\end{pmatrix}
\begin{pmatrix}
r_{1,1}^{(1)}(\xi) \\   r_{1,2}^{(1)}(\xi)
\end{pmatrix} 
- \begin{pmatrix}
0 & 0 & 0\\ 0 & 0 & 0 \\ 0 & 0 & -1
\end{pmatrix}
\begin{pmatrix}
0 \\  0 \\ f_1(\xi)
\end{pmatrix} = 
    \begin{pmatrix} r^{(1)}_{1,1}(\xi) \\ r^{(1)}_{1,2}(\xi) \\ f_1(\xi) \end{pmatrix} &= \begin{pmatrix} -a_2 \xi_1 \xi_2 \\ -a_4 \xi_2 \\ 1 \end{pmatrix}.
\end{align*}
Thus,
\begin{align*}
\phi^{(0)}(\xi_1,\xi_2, \eps) = \phi^{(0)}_0(\xi_1,\xi_2) + \eps \phi^{(0)}_1(\xi_1,\xi_2) + \mathcal{O}(\eps^2) &= \begin{pmatrix}  \xi_1 \\ \xi_2 \\ 0\end{pmatrix} + \eps\begin{pmatrix} 0 \\ 0 \\ 1 \end{pmatrix} + \mathcal{O}(\eps^2)
\end{align*}
and
\begin{align*}
r^{(1)}(\xi, \eps) =  \eps \begin{pmatrix} -a_2 \xi_1 \xi_2 \\ -a_4 \xi_2 \end{pmatrix} + \mathcal{O}(\eps^2) = \eps  \begin{pmatrix} -a_2 \xi_1 \\ -a_4 \end{pmatrix} \xi_2 + \mathcal{O}(\eps^2)\, .
\end{align*}
This agrees with the computation of the slow vector field in \eqref{baby-slow}. Using Remark \ref{remk:Nfsplitting} again, we read off the existence of a further critical manifold
\begin{align*}
S_1 &= \{(\xi_1,\xi_2): \xi_2 = 0\},
\end{align*}
admitting a smooth parametrisation $\phi_0^{(1)}: U_2 := \mathbb{R} \to U_1$:
\begin{align*}
\phi^{(1)}_0(\eta) &= \begin{pmatrix} \eta \\ 0 \end{pmatrix}.
\end{align*}
Using Remark \ref{remk:Nfsplitting2} again, we note that the fast fibres along $S_1$ are spanned by 
\begin{align*}
N^{(1)}_0(\eta) &= \begin{pmatrix} -a_2 \eta \\ -a_4 \end{pmatrix}.
\end{align*}
We now study the terms in $G_2$ using \eqref{eq:G2}. Clearly $-D\phi^{(0)}_1 r^{(0)}_1 = 0 $ and $F_2(\phi^{(0)}_0) = 0$ on $S_1$. The remaining terms are
\begin{align*}
DF_1(\phi^{(0)}_0)\cdot \phi^{(0)}_1 &= \begin{pmatrix} -a_2 \xi_2 & -a_2 \xi_1 & a_1 \\ 0 & -a_4 & a_3 \\ 0 & 0 & 0  \end{pmatrix} \begin{pmatrix} 0 \\ 0 \\ 1 \end{pmatrix} = \begin{pmatrix} a_1 \\ a_3 \\ 0 \end{pmatrix}\ \mbox{and}  \ 
\frac{1}{2} D^2F_0(\phi^{(0)}_0)(\phi^{(0)}_1,\phi^{(0)}_1) = 0\, .
\end{align*}
 Note the nontrivial contribution of $DF_1(\phi^{(0)}_0)\cdot \phi^{(0)}_1$. It remains to compute $r_2^{(2)}$, using \eqref{formular22}. Using the expressions we found for $D\phi^{(0)}_0,\,D\phi^{(1)}_0,\,N^{(0)}_0,$ and $N^{(1)}_0$, one computes that the auxiliary projection operators -- see \eqref{eq:tildeP} -- are given by
 \begin{align*}
\tilde{P}_0^{(0)}(\xi) &= \begin{pmatrix} 1 & 0 & 0 \\ 0 & 1 & 0 \end{pmatrix}\, ,\\
\tilde{P}_0^{(1)}(\eta) &= \begin{pmatrix} 1 & -\dfrac{a_2}{a_3}\eta  \end{pmatrix}.
\end{align*}
Denoting $\xi=\phi_0^{(1)}(\eta)$, the leading-order term of the infra-slow vector field on the chart $U_2$ is
\begin{align*}
r_2^{(2)}(\eta) &= \tilde{P}_0^{(1)}(\eta)\tilde{P}_0^{(0)}(\xi) G_2(\xi) \\
&= a_1 - \frac{a_2 a_3}{ a_4} \eta, 
\end{align*}
i.e. the asymptotic expansion of the infra-slow vector field on $U_2$  is
\begin{align*}
r^{(2)}(\eta, \eps) = \left( a_1 - \frac{a_2 a_3}{ a_4} \eta \right) \eps^2 +\mathcal{O}(\eps^3)\, .
\end{align*}
This result agrees with that in \eqref{eq:babyinfra-slow}. It is not hard to show that $r^{(2)}_{i} =0$ for $i=3, 4, \ldots$, in agreement with the exact expression for $r^{(2)}$ given in the introduction. From the leading-order part of the infra-slow vector field we also recover the equilibrium point
\begin{align*}
\eta^* &= \frac{a_1 a_4}{a_2 a_3}\, .
\end{align*}
Note that $(\phi^{(0)}_0\circ \phi^{(1)}_0)(\eta^*) = (\eta^*, 0, 0)$ is $\mathcal{O}(\eps)$ away from $\left( \frac{a_1 a_4}{a_2 a_3}, \eps \frac{a_3}{a_4}, \eps \right)$, the true equilibrium point of the original system.
\end{ex}

 \begin{ex}[{\it Nontrivial contribution of $\frac{1}{2}D^2 F_0(\phi^{(0)}_0)(\phi^{(0)}_1,\phi^{(0)}_1)$}]
Consider the system
\begin{align} \label{eq:d2f}
\left( \begin{array}{c}  x_1'\\  x_2'\\   x_3' \end{array} \right) &=  F_0 + \varepsilon F_1 = 
\left( \begin{array}{c} x_3^2 \\ 0 \\ -x_3 \end{array} \right)  + \varepsilon   \left( \begin{array}{c} 0 \\ -x_2 \\ 1 \end{array}     \right)
\end{align}
with $0 \leq \eps \ll 1$. 
Like the IFFL model, this system is exactly solvable: it has an attracting invariant slow manifold $S^{\eps}_0=\{x_3 = \varepsilon\}$, on which the dynamics for $(x_1,x_2)$ is governed by
\begin{align*}
\begin{pmatrix} x_1' \\ x_2' \end{pmatrix} &= \eps \begin{pmatrix} 0 \\ -x_2 \end{pmatrix} + \eps^2 \begin{pmatrix} 1 \\ 0 \end{pmatrix}\, .
\end{align*}
This slow vector field in turn has an attracting one-dimensional critical manifold given by $S_1=\{x_2=0\}$, on which the dynamics for $x_1$ is governed by
$$x_1' = \varepsilon^2.  $$
As in the previous example, we now show how the parametrisation method recovers the same slow and infra-slow dynamics. First, note that the critical manifold of the system is given by
 \begin{align*}
S_0 &= \{(x_1,x_2,x_3): x_3 = 0\},
\end{align*}
admitting a smooth parametrisation  $\phi^{(0)}_0: U_1:=\mathbb{R}^2 \to U_0:=\mathbb{R}^3$:
\begin{align*}
\phi^{(0)}_0(\xi_1,\xi_2) &= \begin{pmatrix} \xi_1 \\ \xi_2 \\ 0 \end{pmatrix}.
\end{align*}
We have furthermore that
\begin{align*}
N_0^{(0)}(\xi_1,\xi_2) &= \begin{pmatrix} 0 \\ 0 \\ -1 \end{pmatrix}
\end{align*}
spans the fast fibres with basepoints in $S_0$ (by Remark \ref{remk:Nfsplitting2}). As before, we choose the graph ansatz
\begin{align*}
\phi^{(0)}(\xi_1,\xi_2, \eps) &= \begin{pmatrix} \xi_1 \\ \xi_2 \\ 0 \end{pmatrix} + \eps \begin{pmatrix} 0 \\ 0 \\ f_1(\xi_1,\xi_2)  \end{pmatrix} + \eps^2 \begin{pmatrix} 0 \\ 0 \\ f_2(\xi_1,\xi_2) \end{pmatrix} + \cdots.
\end{align*}
Writing again $r^{(1)}_1 = (r^{(1)}_{1,1}, r^{(1)}_{1,2})$ and realising that $\phi^{(0)}_1 = (0, 0, f_1)$, the first infinitesimal conjugacy equation 
 reduces to  
 \begin{align*}
\begin{pmatrix}
1 & 0 \\ 0 & 1 \\ 0 & 0 
\end{pmatrix}
\begin{pmatrix}
r_{1,1}^{(1)}(\xi) \\   r_{1,2}^{(1)}(\xi)
\end{pmatrix} 
- \begin{pmatrix}
0 & 0 & 0\\ 0 & 0 & 0 \\ 0 & 0 & -1
\end{pmatrix}
\begin{pmatrix}
0 \\  0 \\ f_1(\xi)
\end{pmatrix} = 
    \begin{pmatrix} r^{(1)}_{1,1}(\xi) \\ r^{(1)}_{1,2}(\xi) \\ f_1(\xi) \end{pmatrix} &= \begin{pmatrix}  0 \\ - \xi_2  \\ 1 \end{pmatrix}.
\end{align*}
Thus,
\begin{align*}
\phi^{(0)}(\xi_1,\xi_2, \eps) = \phi^{(0)}_0(\xi_1,\xi_2) + \eps \phi^{(0)}_1(\xi_1,\xi_2) + \mathcal{O}(\eps^2) &= \begin{pmatrix}  \xi_1 \\ \xi_2 \\ 0\end{pmatrix} + \eps\begin{pmatrix} 0 \\ 0 \\ 1 \end{pmatrix} + \mathcal{O}(\eps^2)\, ,
\end{align*}
and
\begin{align*}
r^{(1)}(\xi, \eps) &= \eps \begin{pmatrix}  0 \\ 1 \end{pmatrix}(-\xi_2) + \mathcal{O}(\eps^2)\, .
\end{align*}
We now read off a further critical manifold from the factorisation of the leading-order part of this vector field:
\begin{align*}
S_1 &= \{(\xi_1,\xi_2): \xi_2 = 0\}.
\end{align*}
The manifold $S_1$ admits a smooth parametrisation $\phi^{(1)}_0: U_2:=\mathbb{R} \to U_1$:
\begin{align*}
\phi^{(1)}_0(\eta) &= \begin{pmatrix} \eta \\ 0 \end{pmatrix}.
\end{align*}
Recalling Remark \ref{remk:Nfsplitting2} again, we note that the corresponding fast fibres along $S_1$ are spanned by 
\begin{align*}
N^{(1)}_0(\eta) &= \begin{pmatrix} 0 \\ 1 \end{pmatrix}.
\end{align*}
We proceed to study the terms in $G_2$ using \eqref{eq:G2}. Clearly $-D\phi^{(0)}_1 r^{(0)}_1 = 0 $ and $F_2(\phi^{(0)}_0) = 0$ on $S_1$. The remaining terms are
\begin{align*}
DF_1(\phi^{(0)}_0)\cdot \phi^{(0)}_1 &= 
\begin{pmatrix} 
0 & 0 & 0 \\
0 & -1 & 0\\
0 & 0 & 0  \end{pmatrix}
 \begin{pmatrix} 0 \\ 0 \\ 1 \end{pmatrix}  = \begin{pmatrix} 0 \\ 0 \\ 0 \end{pmatrix} \ \mbox{and}\\
\frac{1}{2}D^2F_0(\phi^{(0)}_0)(\phi^{(0)}_1,\phi^{(0)}_1) &=\frac{1}{2} \begin{pmatrix}  (\phi^{(0)}_1)^* HF_{0,1}  \phi^{(0)}_1\\  (\phi^{(0)}_1)^* HF_{0,2}  \phi^{(0)}_1 \\  (\phi^{(0)}_1)^* HF_{0,3}  \phi^{(0)}_1\end{pmatrix}
 = \frac{1}{2} \begin{pmatrix} 2 \\ 0 \\ 0  \end{pmatrix} = \begin{pmatrix} 1 \\ 0 \\ 0 \end{pmatrix},
\end{align*}
 where $F_0 = (F_{0,1},F_{0,2},F_{0,3})$ and $HF_{0,i}$ denotes the Hessian matrix of $F_{0,i}$. We highlight that the only nontrivial contributions to $G_2$ now come from  second-derivatives of $F_0$.
 
  It remains to compute $r_2^{(2)}$, using \eqref{formular22}. Using the expressions that we found for $D\phi^{(0)}_0,\,D\phi^{(1)}_0,\,N^{(0)}_0,$ and $N^{(1)}_0$, we compute that the auxiliary projection operators -- see \eqref{eq:tildeP} -- are given by
 \begin{align*}
\tilde{P}_0^{(0)}(\xi) &= \begin{pmatrix} 1 & 0 & 0 \\ 0 & 1 & 0 \end{pmatrix}\, ,\\
\tilde{P}_0^{(1)}(\eta) &= \begin{pmatrix} 1 & 0  \end{pmatrix}\, ,
\end{align*}
and thus,
\begin{align*}
r_2^{(2)}(\eta) &= \tilde{P}_0^{(1)}(\eta)\tilde{P}_0^{(0)}(\xi) G_2(\xi) = 1 \, , 
\end{align*}
which implies that
\begin{align*}
r^{(2)}(\eta, \eps) =  \eps^2 + \mathcal{O}(\eps^3)\, .
\end{align*}
Again, it is not hard to show that $r^{(2)}_{i} =0$ for $i=3,4, \ldots$, in accordance with the exact evolution equation $x_2'=\eps^2$ derived above.
  \end{ex}
  \begin{remk}
In both examples of this section, it is useful to interpret $x_3$ as a quasi-forcing term, which introduces an extra timescale into the  system. In fact, more than one extra timescale can arise from one forcing term. Consider the simple system
\begin{align}
x_1'  =& -x_n^{n-1} \nonumber \\
\cdots & \nonumber\\
x_i' =& -x_n^{n-i} x_i\nonumber \text{ \hspace{0.3cm} for } i = 2,\cdots,n-1 
\\
\cdots &\nonumber \\
x_n' =& -x_n + \eps \nonumber
\end{align}
with $0 < \eps \ll 1$. Here, $x_n$ acts as the forcing term. Hidden on the invariant manifold $\{x_n = \eps\}$ is a multiple timescale system in standard form, possessing a further $n-1$ timescales. 
 \end{remk}

 \section{Applying the method to chemical reaction networks}\label{secappl}
A rich source of examples is given by chemical reaction networks of the form
\begin{align} \label{eq:chem}
x' &= \mathbf{N}_0 \cdot f_0(x) + \eps_1 \mathbf{N}_1 \cdot f_1(x) + \eps_2 \mathbf{N}_2 \cdot f_2(x) \cdots + \eps_l \mathbf{N}_l \cdot f_l(x)\, .
\end{align}
Here, {we assume that the `stoichiometric'} matrices $\mathbf{N}_i$ are 
{\it constant} and have full rank, and $1 \gg \eps_1 \gg \cdots \gg  \eps_l$ are dependent or independent small parameters.

In this section we apply the parametrisation method to two chemical reaction networks of the form \eqref{eq:chem}. 
For more information on systems of this form  we refer to the discussion of the mathematical formalism of two-timescale chemical reaction networks with quasi-steady states in \cite{heinrich,stiefenhofer}, and the multiple timescale formulation suggested in \cite{lee2009}. As noted in \cite{feliu2020}, several strategies have been developed to calculate explicit parametrisations of the critical manifolds found in such networks. This suggests that the parametrisation method is well-suited to analyse reaction networks of the general form \eqref{eq:chem}.

Before presenting the examples, we first make a   general observation. In Section \ref{sec41} we study a system of the more special form
\begin{align} \label{eq:chem2}
x' &= \mathbf{N}_0 \cdot f_0(x) + \eps \mathbf{N}_1\cdot f_1(x)\, ,
\end{align}
where $\mathbf{N}_0$ and $\mathbf{N}_1$ both have dimensions $n\times 1$ and $f_0$ and $f_1$ are scalar functions. As usual, we will assume that \eqref{eq:chem2}  possesses an embedded normally hyperbolic critical manifold $S_0 = \phi^{(0)}_0(U_1) = \{x\in \mathbb{R}^n\, |\, f_0(x) = 0\}$ (see Remark \ref{remk:Nfsplitting}), and define  
$$S_1 := \left\{\xi \in U_1 \, \left|\, r_1^{(1)}(\xi) = 0 \right. \right\} \subset U_1\, .$$
The following proposition implies that  \eqref{eq:chem2} cannot support infra-slow dynamics.
\begin{prop}\label{noinfra-slow}
Consider a system of the form \eqref{eq:chem2} where ${\bf N}_0$ and ${\bf N}_1$ are constant $n\times 1$ matrices and $f_0: \mathbb{R}^n\to\mathbb{R}$ and $f_1: \mathbb{R}^n\to\mathbb{R}$ are scalar functions. \\ \indent
Then either $S_0=\phi_0^{(0)}(S_1)$ or $\phi_0^{(0)}(S_1)$ consists entirely of critical points of \eqref{eq:chem2}. 
\end{prop}
\begin{proof}
Recall that the first equation in \eqref{eq:G-recursive} together with formula \eqref{eq:ri} implies that
\begin{align*}
r^{(1)}_1(\xi) &= \underbrace{ (D\phi^{(0)}_0(\xi)^* D\phi^{(0)}_0(\xi))^{-1} D\phi^{(0)}_0(\xi)^*}_{\rm injective} P_0(\xi) \underbrace{(\mathbf{N}_1 \cdot f_1(\phi^{(0)}_0(\xi)))}_{F_1(\phi^{(0)}_0(\xi))}\, .
\end{align*}
Let us now define a partition $$S_1 = Z_1 \cup Z_2\subset U_1$$ as follows:
\begin{align*}
Z_1   := & \{\xi \in U_1 \, | \, F_1(\phi^{(0)}_0(\xi)) = 0\} = \{\xi \in U_1 \, |\,  \mathbf{N}_1 \cdot f_1(\phi^{(0)}_0(\xi)) = 0\}\, , 
\\
Z_2 := & S_1 \setminus Z_1 \, .
\end{align*}
The set $Z_1$ corresponds to true equilibria of \eqref{eq:chem2}, i.e., $F_0(x) = F_1(x)= 0$ for $x=\phi^{(0)}_0(\xi)$ and $\xi \in Z_1$. 
The set $Z_2$, on the other hand, corresponds to points $x=\phi_0^{(0)}(\xi)$ for which the projection $P_0(\xi) F_1(x)$ vanishes while $F_1(x)$ itself does not. 

If $Z_2 = \emptyset$, then $\phi_0^{(0)}(S_1)$ consists entirely of critical points of \eqref{eq:chem2}. 
If $Z_2\neq \emptyset$, then (because $f_1$ is scalar) it follows that
$$P_0(\xi) \mathbf{N}_1 = 0 \ \mbox{for some} \ \xi\in Z_2\, .$$
Now note that the projection $P_0(\xi)$ has constant kernel $\text{im}(\mathbf{N}_0)$, so this proves that $$\text{im}(\mathbf{N}_1) \subset \text{im}(\mathbf{N}_0)\, .$$ 
This in turn implies that $P_0(\xi){\bf N}_1 = 0$ for all $\xi \in U_1$, and therefore that $r^{(1)}_1(\xi) = 0$ for all $\xi \in U_1$. In other words, $S_1=U_1$ and therefore $S_0=\phi_0^{(0)}(S_1)$.  
\end{proof}

\noindent Proposition \ref{noinfra-slow} implies that it is not possible for \eqref{eq:chem2} to possess an infra-slow manifold $S_1$, with ${\rm dim}\, S_1 < {\rm dim}\, S_0$, that supports a nontrivial flow. Instead, such an $S_1$ must  entirely consist of equilibria of the original system.

\begin{remk}
The example network \eqref{babyreaction} in the introduction, which possesses a hidden infra-slow timescale in its dynamics, is {\it not} of the form \eqref{eq:chem2}, because $$F_1(x_1,x_2,x_3) = \begin{pmatrix} a_1 x_3 - a_2x_1 x_2  \\
  a_3 x_3 - a_4 x_2\\
1  \end{pmatrix}$$
cannot be factored as $F_1(x) = \mathbf{N}_1 \cdot f_1(x)$ with $\mathbf{N}_1$ a constant $3\times 1$ matrix.
\end{remk}

\noindent In the remainder of Section \ref{secappl} we   consider two hypothetical chemical reaction networks on $\mathbb{R}^3$. The first system has the form \eqref{eq:chem2} and admits a two-dimensional slow invariant manifold $S_0^{\eps}$ that contains a curve of equilibria (i.e., the set $S_1=Z_1$ is nonempty.) The second example is a two-parameter family of singularly perturbed reaction networks of the  more general form \eqref{eq:chem}. It admits both a slow and an infra-slow manifold. In both examples we also demonstrate how to compute a first order expansion of the fast fibre bundles.

\subsection{Hypothetical reaction network with a  curve of equilibria}\label{sec41}

As mentioned in the introduction, Feliu et. al. derived a formula
for the lowest order term $r_1$ of the vector field $r: U_1 \times[0,\eps_0) \to \mathbb{R}^k$ in the local chart $U_1$ for the critical manifold $S_0$ parametrised by an embedding $\phi_0: U_1\to \mathbb{R}^n$. In fact, the {\it reduced system in parametrised form} in Theorem 1 of \cite{feliu2020} is equivalent to formula \eqref{eq:ri} applied to $i=1$.
The authors  of \cite{feliu2020} also applied this formula to several model systems derived from chemical reaction networks. 

We now revisit one of the models presented in \cite{feliu2020}. 
Applying the parametrisation method to this model, we recover the first order approximation of the reduced vector field found in \cite{feliu2020} in the analysis below. We also compute a second order correction. We furthermore compute the first order correction to the fast fibre bundle.
The model under consideration is
a hypothetical 3-species reaction network consisting of a fast reaction
$$
2X_1 + 2X_2 \overset{k_1}{\underset{k_{-1}}{\rightleftharpoons}} 3X_3\, ,
$$
and a slow reaction
$$
X_1 + X_3 \overset{\eps k_2}{\underset{\eps k_{-2}}{\rightleftharpoons}} 2X_2 \, .
$$
%
The corresponding system of differential equations, for $x_i \geq 0$, reads
\begin{align} \label{feliueqns}
\left\{ \begin{array}{ll}
x_1' &= -2k_1 x_1^2 	x_2^2 + 2k_{-1} x_3^3 + \eps(-k_2 x_1 x_3 + k_{-2}x_2^2)\, ,\\
x_2' &= -2k_1 x_1^2 	x_2^2 + 2k_{-1} x_3^3 + \eps(2k_2 x_1 x_3 -2 k_{-2}x_2^2)\, ,\\
x_3' &= ~~3k_1 x_1^2 x_2^2 - 3k_{-1} x_3^3 + \eps(-k_2 x_1 x_3 + k_{-2}x_2^2)\, .\end{array} \right.
\end{align}
 Here, $k_1,k_{-1},k_2,k_{-2} = \mathcal{O}(1)$ are dimensional parameters and $0 < \eps \ll 1$ is a dimensionless small parameter.
%
\noindent We identify that
\begin{align*}
F_0(x) &= \begin{pmatrix} -2k_1 x_1^2 	x_2^2 + 2k_{-1} x_3^3 \\ -2k_1 x_1^2 	x_2^2 + 2k_{-1} x_3^3 \\ 3k_1 x_1^2 x_2^2 - 3k_{-1} x_3^3  \end{pmatrix} = \begin{pmatrix} -2 \\ -2 \\ 3 \end{pmatrix} (k_1 x_1^2 x_2^2 - (k_1/K)x_3^3) \ \   \mbox{and} \\
F_1(x) &= \begin{pmatrix}  -k_2 x_1 x_3 + k_{-2}x_2^2 \\ 2k_2 x_1 x_3 -2 k_{-2}x_2^2\  \\ -k_2 x_1 x_3 + k_{-2}x_2^2\end{pmatrix} = \begin{pmatrix} -1 \\ 2 \\ -1 \end{pmatrix} (k_2 x_1 x_3 - k_{-2} x_2^2)\ ,
\end{align*}
where we define $K := k_1 / k_{-1}$. Note that  $F_0$ admits a two-dimensional critical manifold 
$$S_0 = \{K x_1^2 x_2^2 = x_3^3\}\, .$$ 
Inside $S_0$ we find a further submanifold $T_0 =  \{k_2 x_1 x_3 = k_{-2} x_2^2\} \cap S_0 \subset \mathbb{R}^3$, on which $F_0$ and $F_1$ both vanish. This $T_0$ constitutes a one-dimensional curve of equilibria of  system \eqref{feliueqns}. Figure \ref{fig:feliu} displays various numerically obtained trajectories of system \eqref{feliueqns}. The critical manifold $S_0$ and the curve of equilibria $T_0$ are clearly noticeable.

\begin{remk}
One may observe that \eqref{feliueqns} is a system of the form \eqref{eq:chem2}. 
Proposition \ref{noinfra-slow} then implies that we do not expect to find a hidden timescale corresponding to nontrivial infra-slow flow. Indeed, the dynamics on $T_0$ is trivial, and not of order $\mathcal{O}(\eps^2)$.
\end{remk}

\noindent A parametrisation of the critical manifold $S_0$ is given by
 \begin{align*}
\phi_0(\xi) &= \begin{pmatrix} \xi_1^3 \\ \xi_2^3\\ K^{1/3} \xi_1^2 \xi_2^2 \end{pmatrix}\, .
\end{align*}
Here $\xi = (\xi_1,\xi_2)$ are local coordinates in the chart $U_1 = \{\xi_1, \xi_2 >0\} \subset \mathbb{R}^2$. Moreover, from the factorisation of $F_0$
we see that we may choose as a basis for the fast fibre bundle along the critical manifold the constant vector
\begin{align*}
N_0(\xi) &= \begin{pmatrix} -2 \\ -2 \\ 3 \end{pmatrix}.
\end{align*}

\begin{figure}[t]
  \centering
        \includegraphics[height=0.55\textwidth,width=0.95\textwidth]{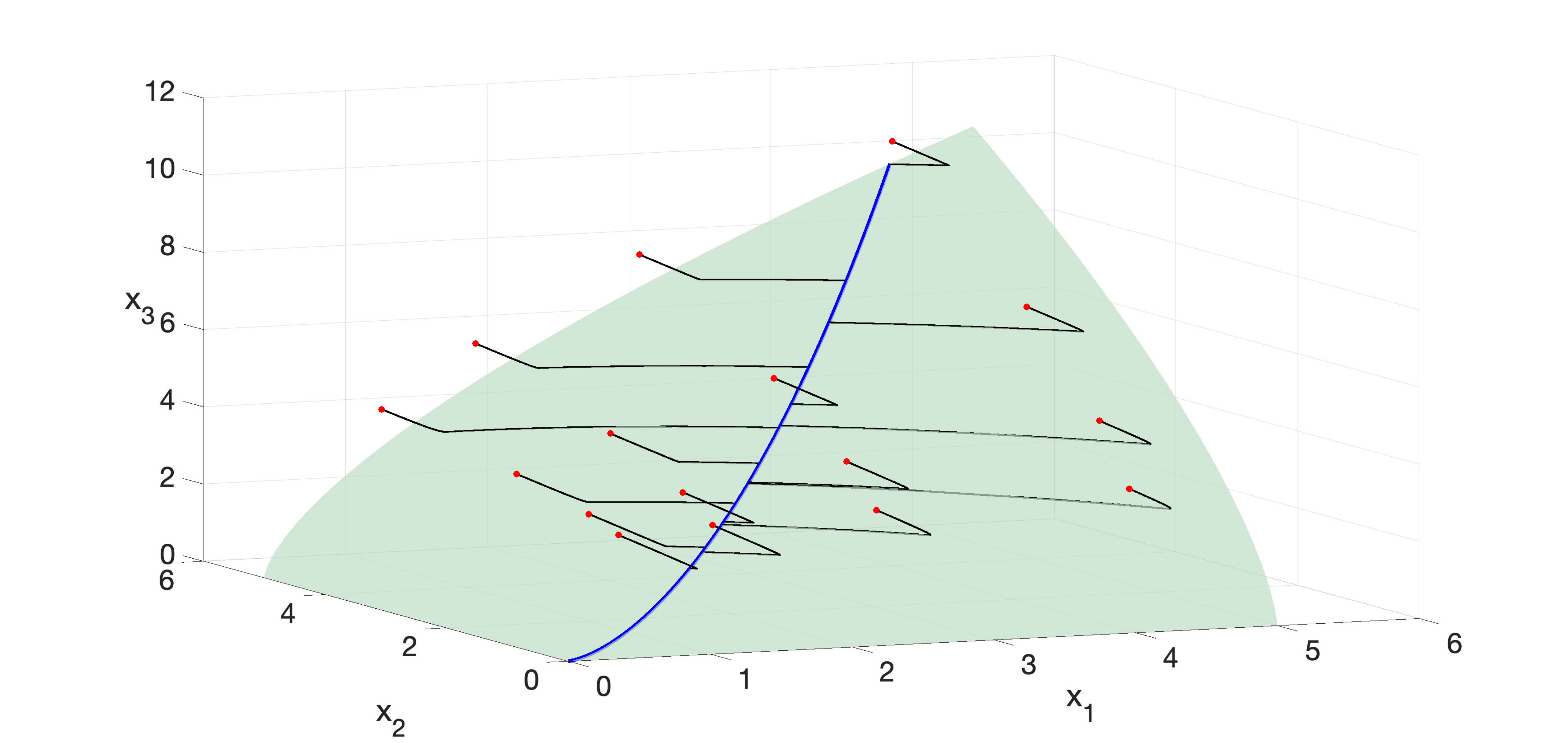}
      \caption{Sample trajectories of \eqref{feliueqns} with initial conditions specified by red points. Green surface: the critical manifold $S_0$. Blue curve: family of equilibria $T_0 \subset S_0$. Parameter set: $\eps = 0.01$, $k_1 = 1.0$, $k_2 = 0.4$, $k_{-1} = 0.5$, $k_{-2} = 0.7$. }  
      \label{fig:feliu}
\end{figure}
\noindent We may now compute $\phi_1$ using \eqref{phiidef} and \eqref{eq:yi}, choosing $X_1 = 0$ for convenience. We find
\begin{align}
\phi_1(\xi) &= N_0(\xi) Y_1(\xi) \nonumber\\
&= -N_0(\xi) n_0(\xi)^{-1} (N_0(\xi)^*N_0(\xi))^{-1}N_0(\xi)^*(1-P_0(\xi))F_1(\phi_0(\xi)) \nonumber 
\\
&= \begin{pmatrix}-2 \\ -2 \\3  \end{pmatrix} \frac{K^{1/3} \left(k_{-2} \xi_2^4 - K^{1/3}k_2\xi_1^5 \right) \left( 3 \xi_1 \xi_2 + K^{1/3} (4\xi_1^3 - 2\xi_2^3) \right)} {k_1 \xi_1^3 \xi_2 \left(9 \xi_1 \xi_2 + 4 K^{1/3} (\xi_1^3 + \xi_2^3)\right)^2}\, . \label{eq:phi1surf}
\end{align}
 %
%
Similarly, the reduced vector field is computed using \eqref{eq:ri}:
\begin{align}
r_1(\xi) &=  (D\phi_0(\xi)^*D\phi_0(\xi))^{-1} D\phi_0(\xi)^* P_0(\xi) (F_1(\phi_0(\xi))) \nonumber \\
&= \frac{K^{1/3} k_2 \xi_1^5 - k_{-2} \xi_2^4}{9 \xi_1 \xi_2 + 4 K^{1/3} (\xi_1^3 + \xi_2^3)} \begin{pmatrix} -(\xi_2^2/\xi_1) (4 K^{1/3} \xi_1^2 + 5 \xi_2) \\ 4 \xi_2 (\xi_1 + K^{1/3} \xi_2^2) \end{pmatrix} \label{eq:r1surf}.
\end{align}
This formula agrees with the result in \cite{feliu2020}. Note that $r_1(\xi) = 0$ along the curve $\{K^{1/3} k_2 \xi_1^5 = k_{-2} \xi_2^4\} = \phi_0^{-1}(T_0) \subset U_1$, corresponding to the curve $T_0 \subset \mathbb{R}^3$ of equilibria of \eqref{feliueqns}.
Higher-order corrections to this reduced vector field must necessarily also remain trivial along this curve of equilibria. Using \eqref{eq:ri} with $i = 2$, \eqref{eq:phi1surf}, and \eqref{eq:r1surf}, we find in particular that
\begin{align*}
r_2(\xi) &= \frac{K^{1/3}(K^{1/3} k_2 \xi_1^5 - k_2 \xi_2^4) \tilde{r}_2(\xi_1,\xi_2)}{k_1\xi_1^3 \xi_2 (9\xi_1 \xi_2 + 4K^{1/3} (\xi_1^3+ \xi_2^3))^3} \begin{pmatrix} -(4K^{1/3} \xi_1^2 + 5 \xi_2)/\xi_1 \\ 4(\xi_1 + K^{1/3}\xi_2^2)/\xi_2 \end{pmatrix},
\end{align*}
where the auxiliary function $\tilde{r}_2(\xi)$ is defined by
$$\tilde{r}_2(\xi) = (-4 k_{-2} \xi_2^3 + 
   k_2 \xi_1^2 (-3 \xi_1 + 2 K^{1/3} \xi_2^2)) (3 \xi_1 \xi_2 + 
   K^{1/3} (4 \xi_1^3 - 2 \xi_2^3)).$$ 
Observe that also $r_2(\xi)$ vanishes along the curve $\{K^{1/3} k_2 \xi_1^5 = k_{-2} \xi_2^4\}=\phi_0^{-1}(T_0) \subset U_1$ (which corresponds to the curve of equilibria of the full system). 

We finish by investigating the linear fast fibre bundle along $S_0^{\eps}$. Here we compute only the first-order correction using \eqref{eq:H-recursive} and \eqref{eq:li}. We find
\begin{align*}
L_1(\xi) &= ((D\phi_0(\xi))^*D\phi_0(\xi))^{-1} (D\phi_0(\xi))^* P_0(\xi) H_1(\xi) n_0(\xi)^{-1} \\
&= \frac{K^{1/3}(-4 k_{-2} \xi_2^3 + k_2 \xi_1^2 (-3 \xi_1 + 2 K^{1/3} \xi_2^2)) }{k_1 \xi_1^3 \xi_2^3 (9 \xi_1 \xi_2 + 4 K^{1/3} (\xi_1^3 + \xi_2^3))^2}  \begin{pmatrix} (1/\xi_1) (4K^{1/3}\xi_1^2 + 5 \xi_2) \\   -(4/\xi_2) (\xi_1 + K^{1/3} \xi_2^2)\end{pmatrix},
\end{align*}
from which it follows that the first-order correction of the parametrisation $N_0(\xi)$  of the fast fibres is given by $\eps N_1(\xi) = \eps D\phi_0(\xi) L_1(\xi)$ (selecting $M_1 = 0$ for convenience). 

\begin{remk}
We performed the calculations of $r_1,\phi_1,r_2,\phi_2$, and $L_1$ in Mathematica. The corresponding Mathematica notebook has been uploaded to  \url{https://github.com/ianlizarraga/ParametrisationMethod}. We also included alternative calculations using the CSP method \cite{lam1994, lam1989,lizarraga2020}.
\end{remk}

\subsection{Three-timescale reaction network.} \label{sec:val}
In this section we study a modified version of a three-species reaction-kinetics model due to Valorani et. al. \cite{valorani2005}. The model describes the three reactions
$$
\begin{aligned}
X_1 & \overset{k_1^f}{\underset{k_{1}^b}{\rightleftharpoons}} 2X_2\\
X_1 + X_2  & \overset{k_2^f}{\underset{k_{2}^b}{\rightleftharpoons}} X_3\\
X_2 + X_3 & \overset{k_3^f}{\underset{k_{3}^b}{\rightleftharpoons}} X_1
\end{aligned}
$$
%
%
%
%
by means of the two-parameter family of differential equations
\begin{align}
\begin{pmatrix} x_1' \\ x_2' \\ x_3' \end{pmatrix} &=
\begin{pmatrix} -1 \\2 \\ 0 \end{pmatrix} (x_1-x_2^2) +
\eps \begin{pmatrix} -1 \\-1 \\ 1 \end{pmatrix} (x_1 x_2-x_3) + \delta \begin{pmatrix} 1 \\ -1 \\ -1 \end{pmatrix} (x_2 x_3-x_1) \label{eq:val}\\
&= F_0(x) + \eps F_1(x) + \delta F_2(x)\,. \nonumber
\end{align}
Here, $x_1, x_2, x_3 \geq 0$, while $\eps\ll 1$ and $\delta\ll 1$ are two small parameters.
\begin{remk}
Valorani et. al. focused their analysis on a rescaled variant of \eqref{eq:val}, and used the CSP method to numerically approximate a one-dimensional attracting slow manifold for a parameter set equivalent to $(\eps,\delta) = (0.2, 0.002)$. They also gave numerical evidence of an intermediate relaxation onto a two-dimensional surface. In \cite{lizarraga2020} a two-timescale subfamily of \eqref{eq:val} (equivalent to choosing $\eps = 0.2$ fixed and letting $\delta$ vary) was analysed from the point of view of GSPT, and formulas for the first-order corrections of the one-dimensional slow manifold and its fast fibres were derived using the CSP method.
\end{remk}
\noindent In the following analysis, we choose $\delta = \bar \delta \eps^2 = O(\eps^2)$ and let $\eps$ vary. This renders the system in the form \eqref{eq:chem} with just one independent small parameter. We show how the parametrisation method algorithmically uncovers genuine three-timescale dynamics (i.e., a nested structure of slow and infra-slow invariant manifolds). Such three-timescale dynamics is also suggested by the numerical integration shown in Figure \ref{fig:valorani1}.  Because our reaction network is entirely hypothetical, we are free to consider other possible relations between $\varepsilon$ and $\delta$ later. 
 



\begin{figure}[t!]
  \centering
        \includegraphics[height=0.4\textwidth,width=0.95\textwidth]{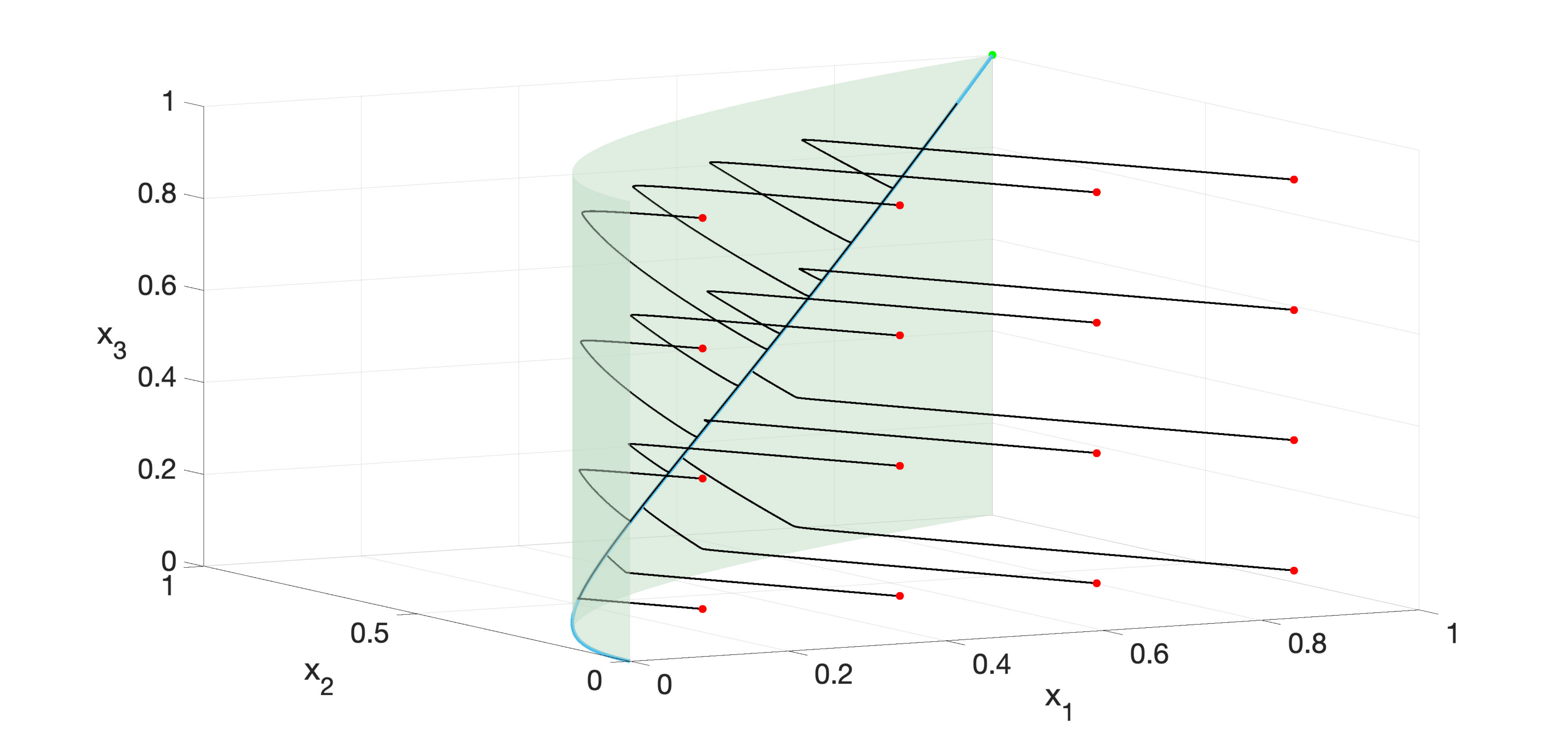}
      \caption{Sample trajectories of \eqref{eq:val} with initial conditions specified by red points. Green surface: the critical manifold $S_0$. Blue curve: the  manifold $T_0 \subset S_0$. Green point: stable node at $(1,1,1)$. Parameter set: $\eps = 10^{-2}$, $\delta = 10^{-4}$. }
      \label{fig:valorani1}
\end{figure}

We start our analysis by noting that our system admits a two-dimensional critical manifold  $$S_0 = \{(x_1,x_2,x_3): x_1 = x_2^2,\,x_2 > 0,\,x_3 > 0\}$$ 
of equilibrium points of $F_0$. This $S_0$ admits a smooth parametrisation
\begin{align*}
\phi^{(0)}_0(\xi) &= \begin{pmatrix} \xi_1^2 \\ \xi_1 \\ \xi_2 \end{pmatrix},
\end{align*}
where $\xi = (\xi_1,\xi_2)$ are local coordinates on a chart $U_1 = \{(\xi_1,\xi_2):\xi_1,\xi_2 > 0\} \subset \mathbb{R}^2$. Note that in this example we use superscripts, in anticipation of finding a nontrivial infra-slow manifold. Recalling Remark \ref{remk:Nfsplitting2}, we also note that the fast fibre bundle along $S_0$ is spanned by
\begin{align*}
N^{(0)}_0(\xi) &= \begin{pmatrix} -1 \\ 2 \\ 0 \end{pmatrix}\, .
\end{align*}
 The first-order corrections $\phi^{(0)}_1$ and $r^{(1)}_1$ are computed in an identical manner to the previous example; selecting $X_1 = 0$ for convenience, and using \eqref{eq:ri}--\eqref{eq:yi}, we find
\begin{align}
\phi^{(0)}_1(\xi) &= N^{(0)}_0(\xi) Y_1(\xi) \nonumber\\
&= -N_0(\xi) n_0(\xi)^{-1} (N_0(\xi)^*N_0(\xi))^{-1}N_0(\xi)^*(1-P_0(\xi))F_1(\phi_0(\xi)) \nonumber\\
&= \frac{1}{(1+4\xi_1)^2} \begin{pmatrix} (\xi_1^3-\xi_2)(2\xi_1-1) \\  -2(\xi_1^3-\xi_2)(2\xi_1-1) \label{eq:valphi1} \\ 0\end{pmatrix}\, ,
\end{align}
%
%
and
\begin{align}
r^{(1)}_1(\xi) &=  (D\phi^{(0)}_0(\xi)^*D\phi^{(0)}_0(\xi))^{-1} D\phi_0^{(0)}(\xi)^* P_0(\xi) (F_1(\phi^{(0)}_0(\xi))) \nonumber\\
&= \begin{pmatrix} \frac{-3}{1+4\xi_1}(\xi_1^3-\xi_2) \\ \xi_1^3 - \xi_2 \end{pmatrix} = \begin{pmatrix} -\frac{3}{1+4\xi_1} \\ 1 \end{pmatrix} (\xi_1^3-\xi_2)\, . \label{eq:valred1}
\end{align}
Note that $r^{(1)}_1(\xi)$ admits a factorisation as  in Remark \ref{remk:Nfsplitting}. 

Similarly, we use \eqref{eq:li} to  compute the first-order correction of the fast fibres. We find
 \begin{align*}
 L_1(\xi) &= \begin{pmatrix}-\frac{3\xi_1(1-2\xi_1)}{(1+4\xi_1)^2} \\ \frac{\xi_1(1-2\xi_1)}{1+4\xi_1}\end{pmatrix}.
 \end{align*}
%
%
Now note that $r^{(1)}_1: U_1 \to \mathbb{R}^2$ admits a curve of  equilibria defined by
$$S_1 = \{\xi_2 = \xi_1^3\} = \left(\phi_0^{(0)}\right)^{-1}(\{F_0 = F_1 = 0\}) \subset U_1\, .$$ 
As an embedding for $S_1$ let us choose $\phi^{(1)}_0: U_2 = \{\eta > 0\} \to \mathbb{R}$ defined by
\begin{align*}
\phi^{(1)}_0(\eta) &=  \begin{pmatrix} \eta \\ \eta^3 \end{pmatrix}\, .
\end{align*}
\begin{figure}[t!]
  \centering
        \includegraphics[height=0.4\textwidth,width=0.95\textwidth]{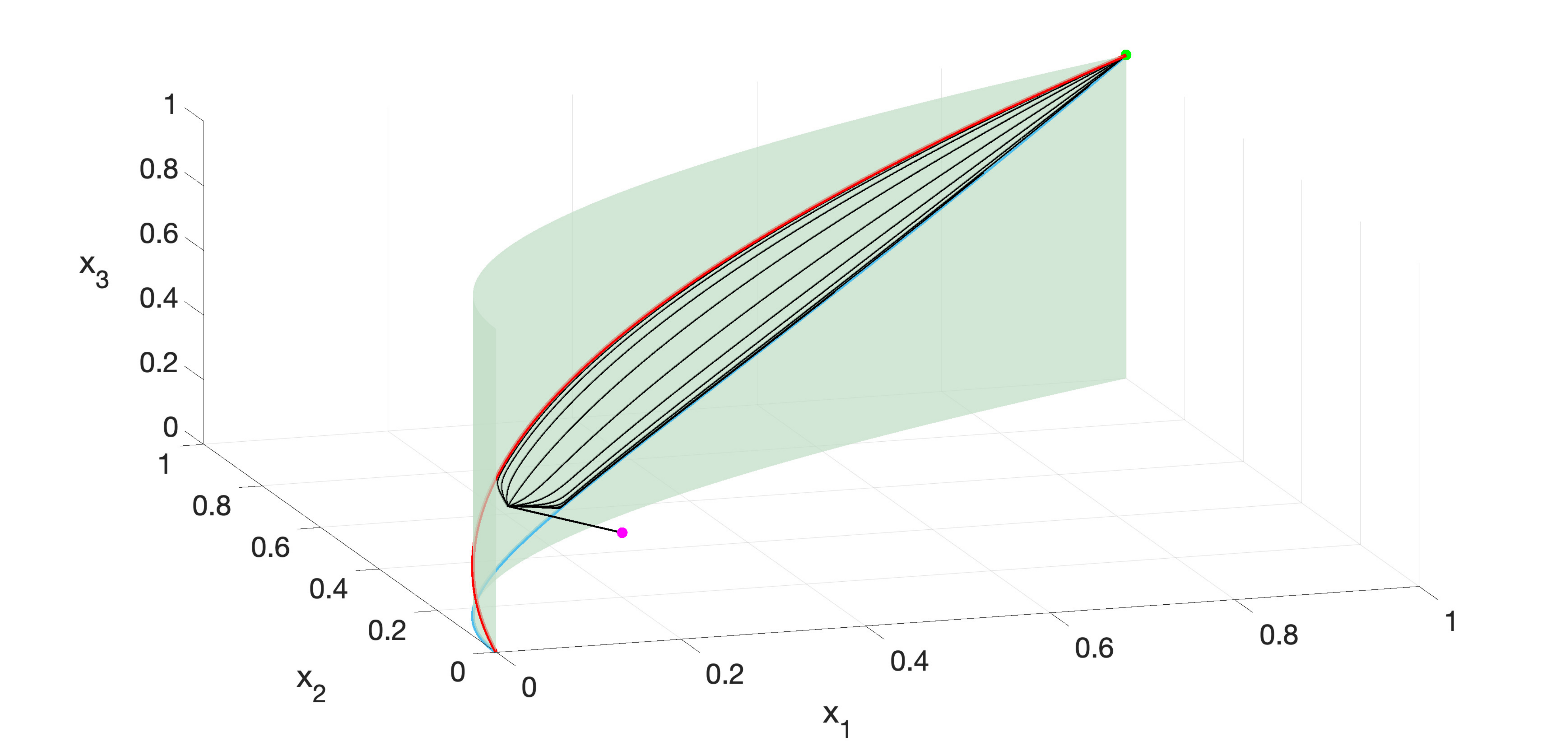}
      \caption{Parameter sweep of  \eqref{eq:val} along the line $(\eps,\delta) = (\eps, (10^{-2}-\eps)+10^{-4})$ for $\eps\in [10^{-4},10^{-2}]$, with initial condition $(0.2,0.2,0.2)$ (magenta point). Green surface is a portion of $S_0 = \{(x_2^2,x_2,x_3): x_2 > 0,\, x_3 > 0\}$, blue curve is $\{(x_2^2,x_2,x_2^3):x_2 >  0\}$, red curve is $\{(x_2^2,x_2,x_2):x_2> 0\}$. Green point: stable node at $(1,1,1)$.}  
      \label{fig:valorani}
\end{figure}
\noindent Using \eqref{eq:valred1} and Remark \ref{remk:Nfsplitting2}, we find that the fast fibre bundle along $S_1$ is spanned by $$N^{(1)}_0(\xi) = \begin{pmatrix} -\frac{3}{1+4\xi_1} \\ 1 \end{pmatrix}.$$ 
We proceed to compute the $O(\eps^2)$ flow on $S_1^{\eps}$ in the chart $U_2$, using  \eqref{formular22}. Proposition \ref{noinfra-slow} implies that a nontrivial term $F_2$ is required to induce a nontrivial flow along a lower-dimensional invariant manifold. In the present example, one can indeed verify directly from \eqref{eq:valphi1} and \eqref{eq:valred1}  that $r^{(1)}_1|_{S_1} = 0$ and $\phi^{(0)}_1|_{S_1} = 0$; it follows by \eqref{formular22} that $F_2\neq 0$ is a necessary condition for $r^{(2)}_2\neq 0$.  From our formulas for $D\phi^{(0)}_0,\,D\phi^{(1)}_0,\,N^{(0)}_0,$ and $N^{(1)}_0$, we obtain that the auxiliary projections defined in \eqref{eq:tildeP} are given by
 \begin{align*}
\tilde{P}_0^{(0)}(\xi) &= \begin{pmatrix}\frac{2}{1+4\xi_1}  & \frac{1}{1+4\xi_1} & 0 \\ 0 & 0 & 1 \end{pmatrix}\, , \\
\tilde{P}_0^{(1)}(\eta) &= \frac{1}{1+4\eta+9\eta^2}\begin{pmatrix} 1+4\eta & 3  \end{pmatrix}.
\end{align*}
By \eqref{formular22}, we therefore have, denoting $\xi=\phi_0^{(1)}(\eta)$,
\begin{align}
r^{(2)}_2(\eta) &= \tilde{P}_0^{(1)}(\eta) \tilde{P}_0^{(0)}(\xi) G_2(\xi) \nonumber\\
&= \tilde{P}_0^{(1)}(\eta) \tilde{P}_0^{(0)}(\xi) F_2(\xi) \nonumber\\
&= \frac{2\bar{\delta}(\eta^2-\eta^4)}{1+4\eta+9\eta^2}. \label{eq:val2}
\end{align}
For $\bar{\delta} = 1$, this expression is identical to the reduced system, written in terms of a parametrisation variable $y$, in \cite{lizarraga2020}. In the latter paper, a two-timescale splitting with a one-dimensional critical manifold $S_0 = \{(x_2^2,x_2,x_2^3): x_2 > 0\}$ was assumed from the start.\\ 
      
\noindent We emphasize that the small parameters $\eps$ and $\delta = \bar{\delta} \eps^2$ are {\it mutually dependent} in the above computations; these computations should therefore be contrasted with the theoretical results in \cite{cardin2017} and \cite{kruff2019}, which are concerned with the persistence of nested invariant manifolds in families of systems with {\it independent} small parameters, such as systems of the form \eqref{eq:cardin1}. Our computations illustrate the utility of the parametrisation method to uncover multiple timescale dynamics when the small parameters are interdependent.

\begin{remk}
After considering the parameter regime $\delta = O(\eps^2)$, let us now briefly discuss how the computations differ in other parameter regimes. For $\eps = O(\delta^2)$, an identical analysis to the above reveals a different three-timescale system, in which the one-dimensional critical manifold of the reduced problem is 
$S_1 = \{(x_2^2, x_2,x_2)\, :\, x_2 >0\}$. Varying $\varepsilon$ and $\delta$ simultaneously along a curve in parameter space connecting the regions $\eps \ll \delta \ll 1$ and $\delta \ll \eps \ll 1$, we can identify these two distinct three-timescale systems as limiting cases of a one-parameter family of two-timescale systems; see Fig. \ref{fig:valorani}. 
Numerical integration suggests that in the intermediate regime $\eps \approx \delta$, trajectories near $S_0$ still relax onto a one-dimensional curve. This relaxation corresponds to the approach of the trajectories to the stable node at $(1,1,1)$ along a weak eigendirection. The Jacobian evaluated at $(1,1,1)$ is 
\begin{align}
DF(1,1,1) &= \begin{pmatrix} -\eps - \delta -1 & -\eps + \delta + 2 & \eps + \delta \\
-\eps + \delta + 2 & -\eps - \delta - 4 & \eps -\delta \\
\eps + \delta & \eps - \delta & -\eps - \delta
\end{pmatrix}\, .
\end{align}
Sylvester's criterion verifies that the equilibrium point is indeed a stable node for all $\varepsilon,\delta > 0$. The corresponding relaxation does not imply the existence of a third (hidden) timescale though, since $r^{(1)}_1$ now does not admit a curve of equilibria.
 \end{remk}
\section{Concluding remarks} \label{secremarks}
\begin{figure}[t]
  \centering
        \includegraphics[height=0.4\textwidth,width=0.8\textwidth]{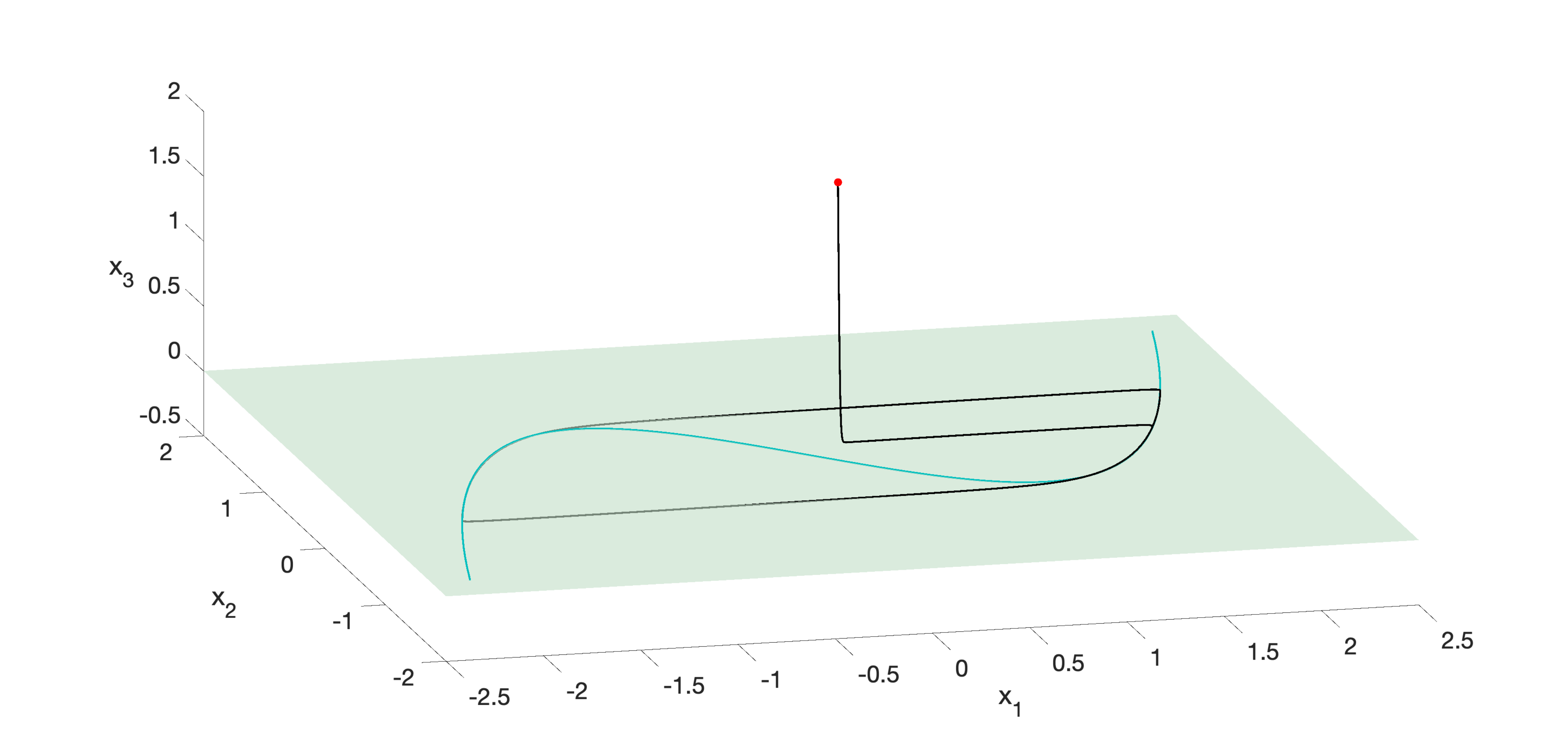}
      \caption{Sample trajectory of \eqref{embeddedvdp} with $\eps = 0.01$ and initial condition $(0.2,0.2,2)$ (red point). Critical manifold $S_0 = \{x_3 = 0\}$ depicted by green surface; infra-slow critical manifold $S_1 = \{(x_1,x_1^3/3-x_1,0): x_1 \in \mathbb{R}\}$ depicted by teal curve.}  
      \label{fig:embeddedvdp}
\end{figure}

We have described a novel method to approximate slow manifolds and fast fibre bundles in multiple timescale systems of the general form \eqref{eq:singpert}  in the spirit of the parametrisation method of Cabr\'e, Fontich and de la Llave et. al. \cite{Cabre1, Cabre2, Cabre3, parameterisation_book}. The power of our method lies in its {\it top-down} approach: after computing a smoothly embedded slow manifold $S_0^{\eps}$, we can algorithmically uncover and approximate its nested infra-slow manifolds and fast fibres. 
We emphasize that no {\it a priori} knowledge of the ultimate attractor state(s) or the number of disparate timescales is needed to apply the method.

We highlight two important applications of the parametrisation method in this paper: we study the phenomenon of {\it hidden} timescales, that can arise in general singularly perturbed systems, and we show how the method can uncover multiple timescale dynamics in systems with {\it dependent} small parameters. Both applications lie beyond the current multiple timescale framework introduced by Cardin \& Teixeira in \cite{cardin2017}, and we conclude that systems of the form  \eqref{eq:cardin1} may be too restrictive for the study of general multiple timescale dynamical systems. Building on the important first results in \cite{cardin2017}, we suggest a (coordinate-independent) geometric definition of a multiple timescale  system in our Definition \ref{def:mult}. The theory for the corresponding nested invariant manifolds in this general context merits further investigation.

One of our goals was to demonstrate the practical implementation of our parametrisation method. Our examples were limited to second-order corrections of the slow manifolds and first-order corrections of the fast fibres, computed symbolically with Mathematica. We anticipate that naive symbolic calculations to obtain higher-order corrections will become unwieldy, since the cost of computing the derivatives in the definitions of $G_i$ and $H_i$---see \eqref{eq:G-recursive} and \eqref{eq:H-recursive}---grows quickly. For more complex problems, we suggest pairing the parametrisation method with automatic differentiation algorithms \cite{autodiff} to handle this inefficiency. Automatic differentiation has been previously adapted to the parametrisation method to efficiently approximate invariant manifolds of fixed points \cite{parameterisation_book} and isochrons of limit cycles \cite{huguet}.

In this paper we focused on the case of nested {\it normally hyperbolic} invariant manifolds. When the nested manifolds are attracting, the long-term behavior of typical trajectories is governed by the dynamics on the slowest timescale. More complicated behavior can occur in general, including switching between several regimes of dynamical timescales as observed in, for example, relaxation oscillations. These more complicated transitions can arise due to possible loss of normal hyperbolicity at some `level' of the nesting, which may not at all be apparent from the general form \eqref{eq:singpert}. The following {\em embedded} Van der Pol system illustrates some of the complications that can occur:
\begin{align} 
x_1' &= \eps (x_2+x_1-x_1^3/3) \nonumber \\
x_2' &= -\eps x_3 \label{embeddedvdp}\\
x_3' &= -x_3 + \eps x_1. \nonumber
\end{align}
This system admits a critical manifold $S_0 = \{x_3 = 0\}$, and we emphasize that the critical manifold at this `top' level is normally hyperbolic attracting. For $0 < \eps \ll 1$, typical trajectories in fact approach a hidden relaxation oscillation within $S_0^{\eps}=\{x_3=\eps x_1 +O(\eps^2)\}$, connecting the slow and the (hidden) infra-slow timescales (see Fig. \ref{fig:embeddedvdp}).

 We suggest that similar scenarios can arise in chemical reaction networks: a subset of fast precursors is rapidly extinguished, after which the system settles onto a low-dimensional attractor which may itself have distinct temporal features.  The parametrisation method can be used to identify both the slow and the infra-slow manifolds and fast fibres, away from the fold points where normal hyperbolicity is lost. In general, the parametrisation method may serve as a useful tool, as a more complete theory for the loss of normal hyperbolicity in multiple timescale systems is developed.
\\

\noindent
{\bf Acknowledgement:}
The authors would like to thank the referees for their valuable feedback. BR is happy to acknowledge the Sydney Mathematical Research Institute for its hospitality and financial support, and the Dutch mathematics cluster NDNS+ for providing travel support. IL and MW would like to acknowledge support through the ARC DP180103022 grant.

\end{document}